\documentclass[a4paper,12pt]{amsart}
\usepackage{amsmath,
			amssymb,
			amsthm,
			amsfonts,
			amscd,
			mathrsfs, 
			enumerate, 
			mathtools} 
\usepackage[utf8]{inputenc}
\usepackage{psfrag}
\usepackage{graphicx}
\usepackage[all]{xy}
\usepackage[normalem]{ulem}
\usepackage{url} 
\usepackage{color}
\usepackage[top=1.5in,
			bottom=1.5in,
			left=1in,
			right=1in]{geometry}
\usepackage{braket} 

\usepackage{etoolbox}
\makeatletter
\patchcmd{\@maketitle}
  {\ifx\@empty\@dedicatory}
  {\ifx\@empty\@date \else {\vskip3ex \centering\footnotesize\@date\par\vskip1ex}\fi
   \ifx\@empty\@dedicatory}
  {}{}
\patchcmd{\@adminfootnotes}
  {\ifx\@empty\@date\else \@footnotetext{\@setdate}\fi}
  {}{}{}
\makeatother

\newdimen\theight
\def\TeXref#1{%
             \leavevmode\vadjust{\setbox0=\hbox{{\tt
                     \quad\quad  {\small \textrm #1}}}%
             \theight=\ht0
             \advance\theight by \lineskip
             \kern -\theight \vbox to
             \theight{\rightline{\rlap{\box0}}%
             \vss}%
             }}%

\DeclareMathOperator{\Diff}{Diff}

\DeclareMathOperator{\depth}{depth}
\DeclareMathOperator{\OO}{O}
\DeclareMathOperator{\codim}{codim}
\DeclareMathOperator{\Hom}{Hom}

\newtheorem{theorem}{Theorem}[section]
\newtheorem{lemma}[theorem]{Lemma}
\newtheorem{proposition}[theorem]{Proposition}
\newtheorem{corollary}[theorem]{Corollary}
\theoremstyle{definition}
\newtheorem{definition}[theorem]{Definition}
\newtheorem{example}[theorem]{Example}
\newtheorem{examples}[theorem]{Examples}

\theoremstyle{remark}
\newtheorem{remark}[theorem]{Remark}
\newtheorem*{org}{Organization of the article}

\numberwithin{equation}{section}

\newcommand\chii{\raise2pt\hbox{$\chi$}}
\newcommand\phii{{\raise2pt\hbox{$\varphi$}}}

\newcommand\F{\mathcal{F}}
\newcommand\E{\mathcal{E}}
\newcommand\K{\mathcal{K}}
\newcommand\G{\mathcal{G}}
\newcommand\I{\mathcal{I}}
\newcommand\HH{\mathcal{H}}

\renewcommand\S{\mathbb{S}}
\newcommand\T{\mathbb{T}}
\newcommand\D{\mathbb{D}}
\newcommand\Z{\mathbb{Z}}
\newcommand{\R}{\mathbb{R}}

\newcommand\fF{\mathscr{F}}
\newcommand\fE{\mathscr{E}}
\newcommand\fH{\mathscr{H}}
\newcommand\fB{\mathscr{B}}
\newcommand{\SK}{\hbox{\sf S}_{\mathcal K}}
\newcommand{\SF}{\hbox{\sf S}_{\mathcal F}}
\newcommand{\SFp}{\hbox{\sf S}_{\mathcal F'}}
\newcommand{\SKDS}{\hbox{\sf S}_{{\mathcal K}_{D_{S}}}}

\newcommand{\SKTS}{\hbox{\sf S}_{{\mathcal K}_{T_{S}}}}
\newcommand{\rondp}{\hbox{\footnotesize $\circ$}}
\newcommand{\kappab}{\hbox{\boldmath{$\kappa$}}}

\newcommand{\qi}{\stackrel{\rm{q.i.}}{\hookrightarrow}}

\newcommand{\pr}{\mathop{\rm pr}}

\title{Cohomological Tautness of Singular Riemannian Foliations}

\author[J.I.~Royo Prieto]{Jos\'{e} Ignacio Royo Prieto}
\address{Department of Applied Mathematics\\ University of the Basque Country UPV/EHU\\ Pza. Ingeniero Torres Quevedo n.~1\\ 48013 Bilbao, Spain.}
\email{joseignacio.royo@ehu.eus}
\thanks{The first author was partially supported by the Spanish MINECO grant MTM2016-77642-C2-1-P and the Gobierno Vasco/Eusko Jaurlaritza Grant IT1094-16}

\author[M.~Saralegi Aranguren]{Martintxo Saralegi-Aranguren}
\address{F\`{e}d\`{e}ration CNRS\\  Nord-Pas-de-Calais FR 2956\\ UPRES-EA 2462 LML\\
Facult\'e Jean Perrin\\ Universit\'{e} d'Artois\\   Rue Jean Souvraz SP 18\\ 62 307 Lens Cedex, France.}
\email{saralegi@euler.univ-artois.fr}
\thanks{The second author was partially supported by the Spanish MINECO grant MTM2016-77642-C2-1-P}

\author[R.~Wolak]{Robert Wolak}
\address{Instytut Matematyki\\ Uniwersytet Jagiellonski\\
	ul. prof. Stanis{\l}awa {\L}ojasiewicza 6
	30-348 Krak\'ow,  Poland.}
\email{robert.wolak@im.uj.edu.pl}



\date{\today}
	
\begin{document}

\begin{abstract}
	For a Riemannian foliation $\F$ on a compact manifold $M$, J.~A.~\'Alvarez L\'opez proved that the geometrical tautness of $\F$, that is, the existence of a Riemannian metric making all the leaves minimal submanifolds of $M$, can be characterized by the vanishing of a basic cohomology class $\kappab_M\in H^1(M/\F)$ (the \'Alvarez class). In this work we generalize this result to the case of a singular Riemannian foliation $\K$ on a compact manifold $X$. In the singular case, no bundle-like metric on $X$ can make all the leaves of $\K$ minimal. In this work, we prove that the \'Alvarez classes of the strata can be glued in a unique global  \'Alvarez class $\kappab_X\in H^1(X/\K)$.
	As a corollary, if $X$ is simply connected, then the restriction of $\K$ to each stratum is geometrically taut, thus generalizing a celebrated result of E.~Ghys for the regular case.
\end{abstract}

\maketitle

{\centering\footnotesize \em Dedicated to Professor Felipe Cano on his 60th birthday.\par}

\setcounter{tocdepth}{1}
\tableofcontents

\section{Introduction}

\subsection{Tautness and cohomology}\label{subsec:tautnesscohomology}
A foliation $\F$ on a manifold $M$ is said to be (geometrically) {\em taut} if there exists a metric such that every leaf of $\F$ is a minimal submanifold of $M$. Tautness is a relevant property of foliations, which has been extensively studied since the eighties and nineties of the last century. It is essentially a transverse property: A.~Haefliger proved that if $M$ is compact, the tautness of $\F$ depends only on its transverse structure, namely, on the holonomy pseudogroup of $\F$ (see \cite[Theorem~4.1]{haefliger80}).

In the case of Riemannian foliations (those admitting a bundle-like metric, that is, a metric whose orthogonal component is holonomy invariant), tautness is remarkably of topological nature, as the following results show. In \cite{masa}, X.~Masa proved that if $\F$ is an oriented and transversally oriented Riemannian foliation, then it is taut if and only if the top degree group of the basic cohomology is isomorphic to $\R$, as conjectured by Y.~Carri\`ere in his Ph.D. Thesis. In \cite{suso}, J.A.~\'Alvarez L\'opez defined the so-called {\em \'{A}lvarez class} (or tautness class) $\kappab_M\in H^1(M/\F)$ whose vanishing characterizes the tautness of a Riemannian foliation $\F$ on a compact manifold $M$. As a corollary, he removed the assumption of orientability of Masa's characterization. Another immediate consequence of \'Alvarez's result is that any Riemannian foliation on a simply connected and compact manifold $M$ is taut, which was proven previously by E.~Ghys in \cite[Th\'eor\`eme B]{ghys}. One more characterization is obtained by combining these results with F.~Kamber and Ph.~Tondeur's Poincar\'e duality property for the basic cohomology ~\cite[Theorem~3.1]{kamber-tondeur-dual}:
\begin{equation}\label{eq:tpdKT}
H^{*}(M/\F)\cong \Hom(H_{\kappa}^{q-*}(M/\F),\R)\;,
\end{equation}
where $q$ is the codimension of $\F$ and the $\kappa$-twisted basic cohomology $H^{*}_{\kappa}(M/\F)$ stands for the cohomology of the basic de Rham complex with the twisted differential $d_{\kappa}\omega=d\omega - \kappa\wedge\omega$. It follows that, under the assumptions of Masa's theorem, $\F$ is taut if and only if $H^0_{\kappa}(M/\F)\cong\R$. For an account of the history of tautness and cohomology of Riemannian foliations see \cite{tope2} and V.~Sergiescu's Appendix in \cite{molino}. 

\subsection{Tautness of strata of singular Riemannian foliations}
We are interested in finding the singular version of those results in the less explored framework of a singular Riemannian foliation (SRF, for short) $\K$ on a compact manifold $X$. We now summarize some results we have obtained in this direction.

One major difference with the regular case is that, in an SRF, geometrical tautness is not to be achieved globally: there exists no bundle-like metric on $X$ making all the leaves of $\K$  minimal submanifolds of $X$ (essentially, because they are of different dimensions, see \cite[p.~186]{tope1} and \cite{miquel-wolak}). It is natural, then, to focus on the foliations $\K_S$ induced by $\K$ on each stratum $S$, which are regular. Notice that strata may not be compact submanifolds of $X$. The mean curvature form for non-compact manifolds has a different behavior: Example 2.4 of \cite{cairns-escobales} describes a Riemannian foliation $\F$ on a non-compact manifold $M$ whose mean curvature form $\kappa$ is basic, but not closed, showing that the \'Alvarez class may not even be defined if $M$ is not compact.

Nevertheless, in \cite{tope1} and \cite{tope2} we proved that, for a certain class of foliations called CERFs, the \'Alvarez class is well defined and  the characterizations of tautness described above hold. CERFs are regular Riemannian foliations on possibly non-compact manifolds that can be suitably embedded in a regular Riemannian foliation on a compact manifold called zipper, and whose basic cohomology is computed by a compact saturated subset called reppiz.

The main point is that the singular strata of an SRF are CERFs. Although the compactness of $X$ does not imply the compactness of the strata, each stratum will inherit the cohomological behavior of tautness from its zipper.  Hence, the rich classical cohomological study of tautness applies to each stratum of any SRF defined on a compact manifold.

\subsection{Main result}
In this work we intend to understand the tautness character of all strata globally, by showing that the topology of $X$ has, indeed, a strong influence on the tautness of each stratum, individually. Our main result is the following:

\begin{theorem}\label{theorem:main}
Let $\K$ be an SRF on a closed manifold $X$. Then  there exists a unique class $\kappab_X\in H^1(X/\K)$ that contains the \'Alvarez class of each stratum. More precisely, the restriction of $\kappab_X$ to each stratum $S$ is the \'Alvarez class of $(S,\K_S)$.
\end{theorem}
We will say that $\kappab_X$ is the {\em\'Alvarez class} of $\K$, and that $\K$ is {\em cohomologically taut} if its \'Alvarez class vanishes. To prove Theorem~\ref{theorem:main} we shall need to exploit the local description of the neighbourhood of a stratum of an SRF, and use strongly the fact that its associated sphere bundle admits a compact structure group. The key technical point needed to patch up the \'Alvarez classes of all strata together is Proposition~\ref{prop:thickasabrick}, which establishes that the \'Alvarez class of a singular stratum $S$ is induced by the \'Alvarez class of a tube along $S$.  As a consequence, although taut and non-taut strata may coexist in an SRF (as it happens in Example~\ref{example:surgery}), all strata below a taut stratum must be taut. As an application, we retrieve a singular version of the classical result by E.~Ghys referred to above:

\begin{corollary}
Every SRF on a compact simply connected manifold $X$ is cohomologically taut.
\end{corollary}

\begin{org}
In Sections~\ref{sec:SRF}  and \ref{sec:CERF} we recall some known facts about SRFs and CERFs, respectively, and prove that certain bundles of singular strata are CERFs. In Section 4 we introduce the notion of thick foliated bundle and study the interplay between the \'Alvarez classes of their components. We apply this study in Section 5, as thick foliated bundles appear in the local structure of an SRF, which we shall use to prove the main results in Section 6. The Appendix is devoted to the reduction of the structure group of the sphere bundle of a stratum.
\end{org}

The authors wish to thank the referees for their useful remarks and suggestions.

\section{Singular Riemannian Foliations}\label{sec:SRF}

In this section we present the class of foliations we are going to study in this paper, which were introduced by P.~Molino in \cite[Chapter 6]{molino}.
They are essentially Riemannian foliations whose leaves may have different dimensions.

\subsection{SRFs}
A {\em Singular Riemannian Foliation} (SRF, for short) on a connected manifold $X$ is a partition $\K$ by connected immersed submanifolds, called {\em leaves}, satisfying the following properties:

\begin{enumerate}
   \item the module of smooth vector fields tangent to the leaves is transitive on each leaf;
   \item there exists a Riemannian metric $\mu$ on $X$, called {\em adapted metric}, such that each geodesic that is perpendicular at one point to a leaf remains perpendicular to every leaf it meets.
\end{enumerate}

The first condition implies that $(X,\mathcal{K})$ is a singular foliation in the sense of \cite{stefan} and \cite{sussmann}. Notice that the
restriction of $\mathcal{K}$ to a saturated open subset $V \subset X$ 	induces an SRF in $V$, which we shall denote by $\K_{V}$.
Any regular Riemannian foliation (RF for short) is an SRF, but the first interesting examples of SRFs are the following:
	
\begin{itemize}
	\item the partition defined by the orbits of an action by isometries of a connected Lie group;
	\item the partition defined by the closures of the leaves of a regular Riemannian foliation;
	\item the partition defined by the closures of the leaves of a singular Riemannian foliation (this is Molino's conjecture, recently proved in  \cite{alexandrino}).
\end{itemize}

\subsection{Stratification}
Let $L_i$ denote the union of all the leaves of $\K$ of dimension $i$. We denote by $\SK$ the stratification of $X$ whose elements, called {\em strata}, are the connected components of the subsets $L_i$, for every $i\ge0$. The restriction of $\mathcal{K}$ to a stratum $S$ is an RF $\mathcal{K}_{S}$. The strata are partially ordered by: $S_{1} \preccurlyeq S_{2}\Leftrightarrow S_{1}\subset \overline{S_{2}}$. Denote by $\prec$ the corresponding strict partial order. For every stratum $S$, we have $\overline{S}=\bigcup_{S_i\preccurlyeq S}S_i$. Thus, the minimal strata are the only closed strata. The maximal stratum, called the {\em regular stratum}, is an open dense subset of $X$, and shall be denoted by $R$. The other strata will be called {\em  singular strata}. Recall that two strata $S$ and $S'$ are {\em comparable} if either $S'\preccurlyeq S$ or $S\preccurlyeq S'$ holds.
	
The  {\em depth}  of a stratum $S \in \SK$, written $ \depth_{\K} S$, is defined to be
the largest $i$ for which there exists a chain of strata $S_i  \prec S_{i-1} \prec \cdots \prec S_0= S$.
So, $ \depth_{\K} S = 0$ if and only if $S$ is a closed stratum.
The depth of $\SK$ is defined as the depth of its regular stratum, and will be denoted by $\depth  \SK$. Notice that $\depth  \SK=0$ if and only if $\K$ is regular.
	
\medskip
	
We now recall some geometrical tools which  we shall use for the study of the SRF $(X,\mathcal{K})$.
	
\subsection{Foliated tubular neighbourhoods}\label{subsec:tubular}

Since a singular stratum $S \in \SK$ is a proper submanifold of $X$, we can consider a tubular neighbourhood  ${\mathcal{T}}_{S}= ({T}_{S},{\tau}_{S},S)$
with fibre the open disk  $\D^{{n}_{S} +1}$.
The following smooth maps are associated with this neighbourhood:
\begin{itemize}
\item The {\em radius map} ${\rho}_{S}\colon {T}_{S}\to [0,1)$
      defined fibrewise by $z\mapsto |z|$. Each $t\not= 0$ is a regular value of
      ${\rho}_{S}$, and we have ${\rho}_{S}^{-1}(0)=S$.
\item The {\em contraction} ${H}_{S} \colon {T}_{S}\times [0,1] \to {T}_{S}$
      defined fibrewise by $(z,r ) \mapsto  r \cdot z$. The restriction
      $({H}_{S})_{_{t}} \colon {T}_{S}\to {T}_{S}$ is an embedding for each
      $t\not= 0$   and $(H_{S})_{_{0}} \equiv {\tau}_{S}$.
\end{itemize}

\noindent These maps satisfy ${\rho}_{S}(r \cdot u) = r {\rho}_{S}(u)$.
This tubular neighbourhood can be chosen to  satisfy the two following important
properties (see \cite[Lemma 6.1]{molino} and \cite[Lemme 1]{boualem-molino}):

\begin{itemize}
  \item Each  $({\rho}_{S}^{-1}(t),\mathcal{K})$ is an SRF, and
  \item Each  $({H}_{S})_{_{t}}\colon ({T}_{S},\mathcal{K}) \to
	       ({T}_{S},\mathcal{K}) $ is a foliated map.
\end{itemize}

\medskip
		
We shall say that ${\mathcal{T}}_{S}$   is a {\em foliated tubular neighbourhood} of $S$.
The {\em core} of $\mathcal{T}_S$ is the hypersurface ${D}_{S} = {\rho}_{S}^{-1}(1/2)$.
The following map
\begin{equation}\label{eq:explosion}
{\mathscr{L}}_{S}\colon ({D}_{S}\times (0,1),\K \times \I) \to (({T}_{S}\backslash S),\K),
\end{equation}
defined by ${\mathscr{L}}_{S}(z,t) = {H}_{S}(z,2t)$, is a foliated diffeomorphism, where $\I$ stands for the foliation of $(0,1)$ by points.

\subsection{Thom--Mather System}

In Section~\ref{sec:local structure} we shall need the foliated tubular neighbourhoods of the strata satisfying certain compatibility conditions. We introduce the following notion, inspired by the abstract stratified objects of \cite{mather,thom}.

\medskip

A family of foliated tubular neighbourhoods
$\mathcal{T} =\{ {\mathcal{T}}_{S }	\ | \ S $ singular stratum$\}$
is a {\em foliated Thom--Mather system} of $(X,\K)$
if the following conditions are satisfied:

\begin{itemize}
    \item[]
    \begin{itemize}
    \item[(TM1)] For each pair of singular strata $S, S'$ we have
$$
{T}_{S} \cap T_{S'} \ne \emptyset  \Longleftrightarrow  
S\text{ and }S'\text{ are comparable.}
$$
\end{itemize}
\end{itemize}
\noindent Let us suppose that $S'\prec S$. The other conditions are:

\begin{itemize}
    \item[]
    \begin{itemize}
\item[(TM2)]
${T}_{S} \cap T_{S'} = \tau_{S}^{-1}({T}_{S'} \cap S)$.

\item[(TM3)]
$ \rho_{S'} = \rho_{S'} \rondp \tau_{S}\hbox{ on }  {T}_{S} \cap T_{S'}.$

\item[(TM4)]
${\rho}_{S} \rondp ({H}_{S'})_t= {\rho}_{S}$, and ${\rho}_{S'} \rondp ({H}_{S})_t= {\rho}_{S'}$ on ${T}_{S} \cap T_{S'}$, for all $t\in(0,1)$.
\end{itemize}
\end{itemize}

This notion was already defined in \cite[Appendix]{tope1}, but without the condition (TM4). In that paper we constructed a collection of tubular neighbourhoods $\mathcal{T}$ satisfying (TM1), (TM2) and (TM3). Let's show that, in turn, it also satisfies the condition (TM4).

\begin{proposition}\label{proposition:tomate}
Let $\K$ be an SRF  defined on a compact manifold $X$.
Then  there exists a foliated Thom--Mather system of $(X,\K)$.
\end{proposition}
\begin{proof}
Consider the collection $\mathcal{T}$ of foliated tubular neighbourhoods constructed in \cite[Appendix]{tope1}, where it is proven to satisfy (TM1), (TM2) and (TM3). Let's see (TM4).

Consider two singular strata $S'\prec S$. For every  $(y,t)\in T_S\cap T_{S'}\times (0,1)$, we have
$$
\rho_{S'}\circ (H_S)_t(y) \stackrel{(TM3)}{=} 
\rho_{S'}\circ\tau_S (H_S)_t(y) =
\rho_{S'}(y),
$$
and so, the latter part of (TM4) follows.

To prove the first part of (TM4) we recall the description of $T_S\cap T_{S'}$ shown in \cite[3.2]{tope1}.
Consider the SRF $(D_{S'},\K_{D_{S'}})$ and notice that $S''=D_{S'}\cap S$ is a stratum of that SRF. The following restriction of $\mathscr{L}_{S'}$ (\ref{eq:explosion}) is, in fact, an isometry:
\begin{equation}\label{eq:exploisom}
{\mathscr{L}}_{S'}\colon ({D}_{S'}\cap T_S\times (0,1),\K \times \I,\mu|_{D_{S'}}+dt^2) \longrightarrow ((T_{S'}\cap{T}_{S}),\K,\mu).
\end{equation}
Using $\mu_{D_{S'}}$ we can take a foliated tubular neighbourhood $(T_{S''},\tau_{S''},S'')$ of $S''$ in $D_{S'}\cap T_S$ (see Figure~\ref{fig:dibujico}). In fact, $T_{S''}={D}_{S'}\cap T_S$.

\begin{figure}[b]
\includegraphics{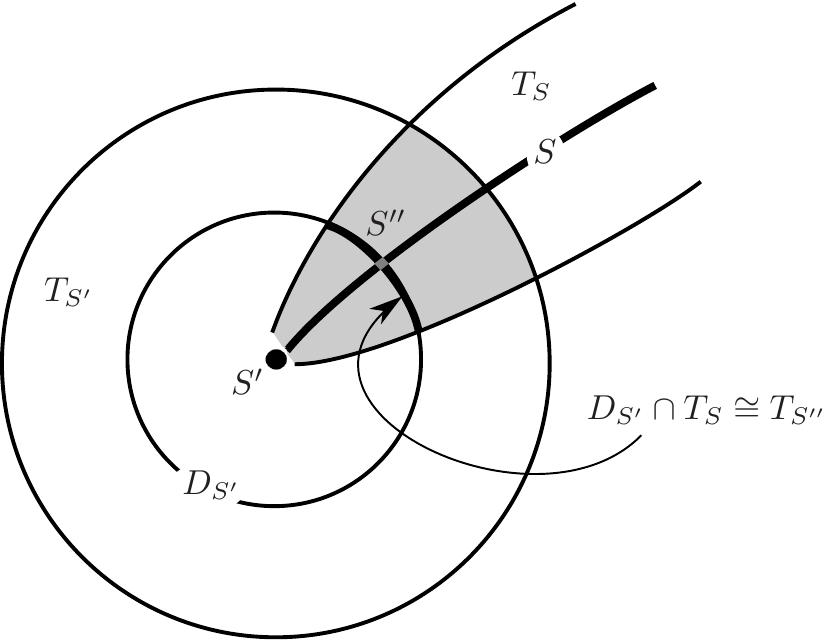}
\caption{The shaded area is $T_{S'}\cap T_{S}\cong T_{S''}\times (0,1)$}
\label{fig:dibujico}
\end{figure}

Now, to prove the first part of (TM4), take $y\in T_{S'}\cap T_S$ and $(x,r)\in T_{S''}\times (0,1)$ such that $\mathscr{L}_{S'}(x,r)=y$. On one hand, notice that \eqref{eq:exploisom} implies that $\rho_{S}(\mathscr{L}_{S'}(x,r))=\rho_{S''}(x)$ for any $r$, and thus, we have $\rho_S(y)=\rho_{S''}(x)$. On the other hand,
\begin{equation*}
\begin{split}
\rho_S\circ (H_{S'})_t (y) 
&= \rho_S\circ (H_{S'})_t (\mathscr{L}_{S'}(x,r)) 
=\rho_S\circ (H_{S'})_t (H_{S'}(x,2r))\\
&= \rho_S\left(H_{S'}(x,2rt)\right)
= \rho_S\left(\mathscr{L}_{S'}(x,rt)\right)
=\rho_{S''}(x),
\end{split}
\end{equation*}
and (TM4) follows.

\end{proof}

Notice that (TM4) does not hold for $S' =S$ since
${\rho}_{S}({H}_{S}(z,r)) = r \cdot {\rho}_{S}(z)$. We fix for the rest of this paper a such
foliated Thom--Mather system  $\mathcal{T}$.

\begin{remark}\label{rem:notrelated}
Given two singular strata $S$ and $S'$, using (TM1) and (TM4), we have:
$$
{T}_{S} \cap {T}_{S'} \ne \emptyset
\Longleftrightarrow
{D}_{S} \cap {T}_{S'} \ne \emptyset
\Longleftrightarrow
{D}_{S} \cap {D}_{S'} \ne \emptyset
\Longleftrightarrow
S\text{ and }S'\text{ are comparable.}
$$
\end{remark}

\subsection{Structure group}\label{subsection.struct.group}
We can take an atlas of the bundle ${\mathcal{T}}_{S}$ whose cocycle takes values in the structure group
$\OO({n}_{S}+1)$. By \cite{molino}, we have that the fibres of this bundle are modelled on an SRF ${\mathcal{E}}_{S}$ on the open disk $\D^{{n}_{S} +1}$.
Moreover, this foliation is invariant by homotheties and the origin is the only $0$-dimensional leaf. This implies that the sphere bundle ${\mathcal{D}}_{S} = ({\tau}_{S} ,{D}_{S} ,S)$ which is, indeed, a restriction of ${\mathcal{T}}_{S}$,  satisfies $\depth \SKDS = \depth\SKTS -1$.

P.~Molino and H.~ Boualem (\cite{boualem-molino})
prove that there exists a foliated atlas
$$
\mathcal{B} =
\Set{\phii_i \colon ({\tau}_{S}^{-1}(U_i) , \mathcal{K})
\to (U_i \times \D^{{n}_{S} +1}, {\mathcal{K}}_{S} \times
{\mathcal{E}}_{S})}
_{i\in I}
$$
of ${\mathcal{T}}_{S}$ whose cocycle takes values in the following structure group
$$
\Diff (\D^{{n}_{S} +1},	 {\mathcal{E}}_{S}) =
\Set{ f \in  \Diff (\D^{{n}_{S} +1}) | f \hbox{
preserves  ${\mathcal{E}}_{S}$ and $|f(v)| = |v|$ if $v \in \D^{{n}_{S} +1}$}}.
$$
We do not know whether this structure group can be reduced to the compact Lie group
$
\OO({{n}_{S} +1}) \cap\Diff (\D^{{n}_{S} +1}, {\mathcal{E}}_{S}),
$
but the following lemma shows that the sphere bundle $\mathcal{D}_S$ has a richer structure group than that of ${\mathcal{T}}_{S}$.
\begin{lemma}\label{lema:struct}
The  sphere bundle ${\mathcal{D}}_{S}$ admits an atlas with values in the compact Lie group
$$
\OO({{n}_{S} +1}, {\mathcal{G}}_{S}) =	
\Set{ A \in  \OO ({{n}_{S} +1}) |  A
\hbox{  preserves } {\mathcal{G}}_{S}},
$$
where ${\mathcal{G}}_{S}$ is an SRF on the fibre ${\S}^{  {n}_{S}  }$ with no $0$-dimensional leaves.
\end{lemma}
\begin{proof}
See Appendix.
\end{proof}

\section{Tautness of Riemannian Foliations}\label{sec:CERF}

In this section we recall that, although a stratum $S$ of an SRF $\K$ is, in general, not compact, the tautness of $\K_S$ can be characterized cohomologically. More precisely, the classical study of tautness applies to a class of RF on possibly non-compact manifolds called CERFs, and $\K_S$ is a CERF.

\subsection{Differential forms}

Let $\F$ be an oriented foliation of dimension $p$ on the Riemannian manifold $(M,\mu)$. The characteristic form $\chi_{\mu}\in\Omega^p(M)$ is defined by
\begin{equation}\label{eq:def chi}
\chi_{\mu}(X_1,\dots,X_p)=\operatorname{det}(\mu(X_i,E_j))\;,\quad\forall X_1,\dots,X_p\in C^{\infty}(TM)\;,
\end{equation}
where $\{E_1,\dots,E_p\}$ is a local oriented orthonormal frame of $T\F$.
The mean curvature form $\kappa_{\mu}\in\Omega^1(M)$ is determined by $\kappa_{\mu}(X)=0$ for all $X\in C^{\infty}(T\F)$ and Rummler's formula~\cite{rummler}:
\begin{equation}\label{eq:rummler}
\kappa_{\mu}(Y)=-d\chi_{\mu}(Y,E_1,\dots,E_p)\;,\quad \forall Y\in C^{\infty}((T\F)^{\perp \mu})\;.
\end{equation}
 Notice that both $\chi_{\mu}$ and $\kappa_{\mu}$ are determined by the orthogonal subbundle $(T\F)^{\perp \mu}$ and the volume form along the leaves. Notice also that $\kappa_{\mu}$ is defined even if no orientation assumptions are made, and it is determined by \eqref{eq:rummler} in an open set where orientation conditions are satisfied.

We say that $\mu$ is {\em taut} (it is also called minimal or harmonic in the literature) if every leaf of $\F$ is a minimal submanifold of $M$, which is tantamount to saying that its mean curvature form is zero. We shall say that $\F$ is {\em taut} if $M$ admits a taut metric with respect to $\F$. If $T\F$ is oriented, tautness is characterized by {\em Rummler-Sullivan's criterion} \cite[Remark in p.~219]{sullivan}, which says that $\F$ is taut if and only if there exists a form $\omega\in\Omega^p(M)$ whose restriction to $T\F$ is positive and such that it is $d\F$-closed; namely, $d\omega(X_1,\dots,X_{p},Y)=0$ for any vector fields $X_1,\dots,X_{p}\in C^{\infty}(T\F)$ and $Y\in C^{\infty}(TM)$.

We say that $\mu$ is {\em tense} if its mean curvature form is basic. We shall say that a tense metric $\mu$ is {\em strongly tense} if $\kappa_{\mu}$ is also closed (see \cite[Def.~2.4]{nozawa-royo}; in \cite{tope1} and \cite{tope2} it is called a {\em D-metric}).

If  $\F$ is an RF on a compact manifold $M$, then there exists a tense metric (\cite[Tenseness Theorem in p.~1239]{dominguez}). As $M$ is compact, any tense metric must also be strongly tense by \cite[Eq.4.4]{kamber-tondeur-basic}.

Strong tenseness is not guaranteed if $M$ is not compact. Example 2.4 of \cite{cairns-escobales} shows an RF of dimension 2 on a noncompact manifold with a tense metric that is not strongly tense. Nevertheless, in \cite[Theorem 1.1]{nozawa-royo} it is shown that any transversely complete RF of dimension 1 on a possibly non-compact manifold $M$ admits a strongly tense metric. In  \cite[Corollary~1.9]{nozawa-arxiv}  the result is extended to any uniform complete RF.

\subsection{The CERFs}\label{subsec:CERF}

Let $M$ be a manifold endowed with an RF $\F$. A {\em zipper} of $\F$ is a compact manifold $N$ endowed with an RF $\HH$ satisfying the following property:
\begin{itemize}
    \item[(a)] The manifold $M$ is a saturated open subset of $N$ and $\HH_{M} = \F$.
\end{itemize}

	\smallskip

\noindent A {\em reppiz} of $\F$  is a saturated open subset  $U$ of $M$ satisfying the following properties:

\begin{itemize}
    \item[(b)] the closure $\overline{U}$ in $M$ is compact;
	\item[(c)] the inclusion $U \hookrightarrow M$ induces the isomorphism $H^{*}(U/\F) \cong H^{*}(M/\F)$,
\end{itemize}

	    \smallskip
\noindent where $H^*(M/\F)$ stands for the {\em basic cohomology}, that is, the cohomology of the complex of basic forms $\Omega^*(M/\F)=\set{\omega\in\Omega^*(M) |  i_X\omega=i_Xd\omega=0\ \forall X\in C^{\infty}(T\F)}$. We shall also use the notation $U\stackrel{q.i.}{\hookrightarrow}M$ to denote a quasi-isomorphism in basic cohomology.

\medskip
	
\noindent   We say that $\F$ is a {\em {\rm C}ompactly {\rm E}mbeddable {\rm R}iemannian {\rm F}oliation } (or	CERF ) if $(M,\F)$ admits a zipper and a reppiz. We have shown in \cite[Proposition 2.4]{tope2} that, for any  stratum $S$ of an SRF $\K$ defined on a compact manifold, the foliation  $\K_{S}$ is a CERF.

\subsection{Tautness of CERFs}\label{subsec:tautCERF}

In \cite[section 3]{tope2} we prove that for a CERF $\F$ on a possibly non-compact manifold $M$ the classical study of cohomology and tautness holds. We summarize the main results here:
\begin{enumerate}[(i)]

\item The CERF $\F$ admits a strongly tense metric $\mu$.
\item The class $[\kappa_{\mu}]\in H^1(M/\F)$ does not depend on the choice of the strongly tense metric $\mu$. We shall call it the {\em \'Alvarez class} of $\F$, and denote it by $\kappab_{\F}$\label{item:kappaCERF}.
\item If $U$ is a saturated open subset of $M$, then the \'Alvarez class of $(U,\F)$ is the restriction of the \'Alvarez class of $(M,\F)$.
\item  The tautness of $\F$ is equivalent to any of the following statements:
\begin{enumerate}
    \item  $\kappab_{\F}= 0$,

    \item  ${H}^{n}(M/\F) \ne 0 $ (when $\F$ is transversally oriented),

    \item $  {H}^{0}_{{\kappa}}(M/\F) \ne 0$ (when  $M$ is oriented and $\F$ is transversally oriented),
\end{enumerate}
\end{enumerate}
where $n=\codim\F$ and the $\kappa$-twisted cohomology ${H}^{*}_{{\kappa}}(M/\F)$ is the cohomology of the complex of basic forms with the twisted differential $d_{\kappa}\omega=d\omega - \kappa\wedge\omega$, being $\kappa\in\Omega^1(M/\F)$ any representative of the \'Alvarez class of $\F$.

\section{Thick foliated bundles}

In this section we slightly generalize the notion of a foliated bundle. 
Principal foliated bundles were introduced by Molino (see \cite[section 2.6]{molino}) to study objects such as the lifted foliation to the transverse frames bundle (see \cite[Proposition 2.4]{molino}). In such bundles,  if a vector  is tangent to a leaf, then it cannot be tangent to the fibre of the bundle
The fibres of a thick foliated bundle may carry a richer foliated structure.
All foliations considered in this section are regular, unless otherwise stated.

\begin{definition}
Let $\fB$, $\fF$ and $\fE$ be foliations on the manifolds $B$, $F$ and $E$, respectively. A fibre bundle $\pi\colon (E,\fE)\longrightarrow (B,\fB)$ with fibre $(F,\fF)$ and structure group $G$ is a {\em thick foliated bundle} if $G\le \Diff{(F,\fF)}$ and there exists an atlas of the bundle $\{(U_i,\psi_i)\}_{i\in I}$ such that the charts
$$
\psi_i\colon\left(\pi^{-1}(U_i),\fE\right) \longrightarrow (U_i\times F,\fB\times\fF)
$$
are foliated diffeomorphisms.
\end{definition}

\begin{example}
Let $\K$ be an SRF over a compact space $X$ and $\pi\colon D_S\longrightarrow S$ be the sphere bundle of a singular stratum $S$, for some Thom--Mather system. The restriction of $\pi$ to the regular part $E=D_S\cap R$ is a thick foliated bundle over $B=S$ whose fibre is the regular part of the model SRF on a sphere (see section~\ref{subsection.struct.group}). The foliations  $\fE$ and $\fB$ are the corresponding restrictions of $\K$. 	
\end{example}

\begin{examples}
The following classical structures are particular cases of thick foliated bundles:
\begin{enumerate}[(i)]
	
	\item If $G=F$, $\pi$ is principal and $\fF$ is the pointwise foliation, then $\pi$ is a principal foliated bundle. 
	
	\item If $\fB$ is the one leaf foliation and $\fF$ is the pointwise foliation, then $\pi$ is a foliated bundle in the sense of \cite{haefliger73}.
	
	\item If $\fB$ is the pointwise foliation, $F$ is compact and $\fF$ is the one leaf foliation, then $\pi$ is a Seifert bundle.
\end{enumerate}	
\end{examples}

If $\pi\colon (E,\fE)\longrightarrow (B,\fB)$ is a thick foliated bundle of fibre $(F,\fF)$, the condition $G\le \Diff{(F,\fF)}$ allows us to define a fibrewise foliation $\fH$ on $E$, just considering
\begin{equation}\label{eq:fibrewise}
\psi_i\colon\left(\pi^{-1}(U_i),\fH\right) \longrightarrow (U_i\times F,\{\text{points}\}\times\fF)
\end{equation}
as foliated diffeomorphisms. We will say that $\fH$ is the {\em fibrewise foliation} associated to $\pi$. The following Lemma shows the interplay between the mean curvature forms of the foliations involved in a thick foliation bundle. The proof is similar to the first part of the proof of \cite[Lemma~7]{nozawa-manuscripta}.

\begin{lemma}\label{lem:thickasabrick}
Let	$\pi\colon (E,\fE)\longrightarrow (B,\fB)$ be a thick foliated bundle. Let $\gamma$ be a metric on $E$ and denote by $\kappa_{\gamma}$ the mean curvature form of $(E,\fH,\gamma)$, where $\fH$ is the associated fibrewise foliation. Then  for any metric $\mu_B$ on $B$, there exists a metric $\mu_E$ on $E$ such that the mean curvature form of $(E,\fE,\mu_E)$ is  
\begin{equation}\label{eq:kappas_bundle}
\kappa_{\mu_E}=(\kappa_{\gamma})_{0,1} + \pi^*\kappa_{\mu_B},
\end{equation}
where 
the bigrading is
associated to $TE=T\fE\oplus (T\fE)^{\bot\gamma}$.
\end{lemma}

\begin{proof}
We consider a bundle decomposition 
$$
TE=
\rlap{$\underbrace{\phantom{\zeta_1\oplus T\fH}}_{\displaystyle{\ker{\pi_*}}}$}
\zeta_1 
\oplus \stackrel{\displaystyle{T\fE}}{\overbrace{T\fH\oplus \zeta_2}} 
\oplus 
\zeta_3\quad
\text{satisfying}
\quad
\begin{minipage}{0.55\textwidth}
\begin{enumerate}[(i)]
\item $\ker{\pi_*}=\zeta_1\oplus T\fH$;\label{item:kerpi}
\item $T\fE=T\fH\oplus\zeta_2$;\label{item:TE}
\item $\zeta_1 \oplus \zeta_3 \le (T\fH)^{\bot\gamma}$;\label{item:orthogonal}
\item $\pi_*\vert_{\zeta_2}\colon\zeta_2\longrightarrow T\fB$ is an isomorphism;\label{item:TB}
\item  $\pi_*\vert_{\zeta_3}\colon\zeta_3\longrightarrow(T\fB)^{\bot\mu_B}$ is an isomorphism.\label{item:TBog}
\end{enumerate}
\end{minipage}
$$
To construct such a decomposition, first, take $\zeta_1= (T\fH)^{\bot\gamma}\cap\ker\pi_*$, which gives \eqref{item:kerpi}. Take $\zeta_2$ any supplementary of $T\fH$ in $T\fE$, thus satisfying \eqref{item:TE}. As  $\pi_*\colon T\fE\to T\fB$ is surjective and $T\fH\le\ker{\pi_*}$, for degree reasons we have \eqref{item:TB}. Now, consider $D$ a supplementary of $\zeta_1$ in $(T\fH)^{\bot\gamma}$. Notice that $\pi_*\vert_D\colon D\to TB$ is an isomorphism. Finally, we take $\zeta_3=(\pi_*\vert_D)^{-1}((T\fB)^{\bot\mu_B})$, which gives \eqref{item:orthogonal} and \eqref{item:TBog}.

Now, we define the metric $\mu_E=\gamma\vert_{\ker\pi_*}+\pi^*\mu_B$ so that the four summands of the previous decomposition are $\mu_E$-orthogonal.
To prove formula \eqref{eq:kappas_bundle} it suffices to check it locally; namely, for any neighbourhood $U\subset E$ small enough so that both $\fH_U$ and $\fE_U$ are oriented (and thus, $\fB_{\pi(U)}$). We now claim that, in $U$, we have
\begin{equation}\label{eq:chiwedge}
\chi_{\mu_E}=\chi_{\gamma}\wedge \pi^*\chi_{\mu_B},
\end{equation}
where $\chi_{\gamma}$ is the characteristic form of $(M,\fH,\gamma)$.
By definition (see \eqref{eq:def chi}) $\chi_{\mu_E}$ is the form that vanishes when applied to any vector orthogonal to $\fE$ and whose restriction to its leaves is a volume form of norm one.

So  consider a $\gamma$-orthonormal and positively oriented parallelism $X_1,\dots,X_p$ of $T\fH_U$ and a $\mu_U$-orthonormal and positively oriented parallelism $Z_1,\dots,Z_q$ of $T\fB\vert_{\pi(U)}$, which are $\pi_*$-related to the $\mu_B$-orthonormal parallelism $Y_1,\dots,Y_q$ of $\zeta_2\vert_U$. As $\pi_*(X_i)=0$ for $i=1,\dots,p$, we have:
\begin{equation*}
\begin{split}
(\chi_{\gamma}\wedge \pi^*\chi_{\mu_B})(X_1,\dots,X_p,Y_1,\dots,Y_q)&=
\chi_{\gamma}(X_1,\dots,X_p) \pi^*\chi_{\mu_B}(Y_1,\dots,Y_q)\\
&=\chi_S(Z_1,\dots,Z_q)=1=\chi_{\mu_E}(X_1,\dots,X_p,Y_1,\dots,Y_q).
\end{split}
\end{equation*}
Now, take $Y$ a section of $(T\fE)^{\bot\mu_E}\vert_U=\zeta_1\vert_U\oplus \zeta_3\vert_U$. By \eqref{item:orthogonal}, we have $Y\in C^{\infty}(T\fH^{\bot\gamma})$, and thus, $i_Y\chi_{\gamma}=0$. We also have  $i_Y(\pi^*\chi_B)=0$, because $\pi_*(Y)\in C^{\infty}((T\fB)^{\bot\mu_B}\vert_{\pi(U)})$ (due to \eqref{item:orthogonal} and  \eqref{item:TBog}). So  \eqref{eq:chiwedge} follows.

To prove \eqref{eq:kappas_bundle}, first take $X$ a section of  $T\fE\stackrel{\eqref{item:TE}}{=}T\fH\oplus\zeta_2$. We have 
$\pi^*(\kappa_{\mu_B})(X)=0$, 
because $\pi_*(X)\in C^{\infty}(T\fB)$, and $(\kappa_{\gamma})_{0,1}(X)=0$ because of the bigrading. So  \eqref{eq:kappas_bundle} holds for $T\fE$.

It remains to check \eqref{eq:kappas_bundle} for $Y$ a section of $(T\fE)^{\bot\mu_E}\vert_U=\zeta_1\vert_U\oplus \zeta_3\vert_U$.
Notice that $\pi_*Y\in C^{\infty}((T\fB)^{\bot\mu_B})$ and $Y\in C^{\infty}((T\fH)^{\bot\gamma})$. So  we finally have:
\begin{equation*}
\begin{split}
\kappa_{\mu_E}(Y)
&\stackrel{\eqref{eq:rummler}}{=}
-d\chi_{\mu_E}(Y,X_1,\dots,X_p,Y_1,\dots,Y_q)\\
&\stackrel{\eqref{eq:chiwedge}}{=}
-d\chi_{\gamma}(Y,X_1,\dots,X_p) \underbrace{\pi^*\chi_{\mu_B}(Y_1,\dots,Y_q)}_{=1}- \underbrace{\chi_{\gamma}(X_1,\dots,X_p)}_{=1}\pi^*d\chi_{\mu_B}(Y,Y_1,\dots,Y_q)\\
&\stackrel{\eqref{eq:rummler}}{=}\kappa_{\gamma}(Y)+\pi^*\kappa_{\mu_B}(Y)=(\kappa_{\gamma})_{0,1}(Y)+\pi^*\kappa_{\mu_B}(Y).\mbox{\qedhere}
\end{split}
\end{equation*}
\end{proof}

\begin{lemma}\label{lema.cover}
	Let $\pi\colon (E,\fE)\longrightarrow (B,\fB)$ be a thick foliated bundle with taut fibre $(F,\fF)$ and compact structure group $G$. Then there exists a foliated Galois $\Z/2\Z$-covering map $\tau\colon (E^{\sharp},\fE^{\sharp})\to (E,\fE)$  such that $\pi^{\sharp}=\pi\circ\tau$ is a thick foliated bundle whose fibre $(F^{\sharp},\fF^{\sharp})$ is taut, and  admitting $G$ as structure group, which preserves a given orientation in $T\fF^{\sharp}$.
\end{lemma}
\begin{proof}
	Consider ${p}\colon (F^{\sharp},\fF^{\sharp})\longrightarrow(F,\fF)$ the tangent orientation covering (see \cite[p.~162]{hectorA}) of $(F,\fF)$, which is a Galois $\Z/2\Z$-covering space such that $T\fF^{\sharp}$ is orientable. Notice that $\fF^{\sharp}$ is taut for the pullback of a taut metric on $F$. For any $f\in\Diff{(F,\fF)}$, we can define $f^{\sharp}\in\Diff{(F^{\sharp},\fF^{\sharp})}$ as the only lifting of $f$ which preserves a given orientation of $T\fF^{\sharp}$. This unicity gives $(Id_{F})^{\sharp}=Id_{F^{\sharp}}$ and $(f\circ g)^{\sharp}=f^{\sharp}\circ g^{\sharp}$. So, we have defined an algebraic group action
	$\Psi_G\colon G\times F^{\sharp}\longrightarrow F^{\sharp}$. 
	
	Let's prove that $\Psi_G$ is, indeed, a smooth Lie group action. It suffices to see that it is continuous, which we will show from its local expression. Without loss of generality, we can take $(U_i,\varphi_i)$ and $(U_j,\varphi_j)$ two foliated charts of $(F,\fF)$ satisfying $f(U_i)\subset U_j$. We then have
	$$
	\xymatrix{U_i\times \Z/2\Z\ar[r]^{f^{\sharp}}\ar[d]_p &
		U_j\times \Z/2\Z\ar[d]_p\\
		U_i\ar[r]^f & 
		U_j, }
	$$
	where the vertical maps are just the projection on the first factor. We have $f^{\sharp}(x,t)=(f(x),\mathop{\rm sign}(|A(x)|))$, being
	$$
	\left(\varphi_j\circ f\circ \varphi_i^{-1}\right)_{*\varphi_i(x)}=
	\left(
	\begin{array}{c|c}
	A(x)& 0\\
	\hline
	B(x)& C(x)
	\end{array}
	\right).
	$$
	We thus get that the correspondence $f\longmapsto f^{\sharp}$ is continuous, and thus $\Psi_G$ is continuous, hence smooth. We just have to change the fibre $(F,\fF)$ by $(F^{\sharp},\fF^{\sharp})$ in the thick foliated bundle $\pi$ to get $\tau$ as desired.
\end{proof}

The proof of the following Lemma is partially similar to that of \cite[Theorem 3.1]{ghys}:

\begin{lemma}\label{lema:particiones de la unidad}
The fibrewise foliation associated to a thick foliated bundle with taut fibre  and compact structure group is taut.
\end{lemma}
\begin{proof}
In \cite[Lemma 6.3]{suso} it is proven that the tautness character of a foliation is preserved by finite coverings (that result is established for  Riemannian foliations on compact manifolds, but the proof does not use those assumptions). So  by Lemma~\ref{lema.cover}, we can assume that the fibre $(F,\fF)$ is tangentially oriented and that the structure group $G$ preserves a given orientation on $T\fF$  to prove that $\fH$ is taut.

Take $\chii\in\Omega^p(F)$ a $d\fF$-closed characteristic form, and consider its averaged form
$$
\overline\chii=\int_G g^*\chii\ dg\in\Omega^p(F),
$$
where  $p=\dim\fF$. We have that $\overline\chii$ is also $d\fF$-closed because the elements of $G$ preserve $\fF$. As they also preserve the orientation of $T\fF$, then the restriction $\overline{\chii}\vert_{T\fF}$ is positive. 

 We can take $\{(U_i,\psi_i)\}_{i\in I}$ an atlas of the thick fibre bundle $\pi$ whose cocycle belongs to $G$, and a partition of unity $\{(U_i,\rho_i)\}_{i\in I}$. For all $i\in I$, define $\phi_i$ and $r_i$ so that the following diagrams are commutative:
$$
\xymatrix{
[0,1] & \pi^{-1}(U_i)
		\ar[l]_{r_i}
		\ar[d]^{\psi_i}
		\ar[dr]^{\phi_i} & \\
U_i \ar[u]^{\rho_i} & 
	U_i\times F
		\ar[l]^{p_i}
		\ar[r]_(0.6){(p_F)_i} & F
}
$$ 
where $p_i$ and $(p_F)_i$ are the projections onto each factor. Now, we define
$$
\chi=\sum_i r_i\cdot  
\phi_i^*
\overline{\chi}
\in\Omega^p(E),
$$
being $E$ the total space of the thick fibre bundle. We now prove that the associated fibrewise foliation $\fH$ is taut by showing that $\chi$ satisfies the conditions of the Rummler-Sullivan criterion. First, we have that $\chi\vert_{T\fH}$ is positive  because $\overline\chi\vert_{T\fF}$ is positive and the maps $\phi_i\colon(\pi^{-1}(U_i),\fH)\to (F,\fF)$ are foliated (see \eqref{eq:fibrewise}). It remains to see that $\chi$ is $d\fH$-closed. Take $X_1,\dots,X_p\in C^{\infty}(T\fH)$  and $Y\in C^{\infty}(TE)$. On one hand, for all $i\in I$ and $j\in\{1,\dots,p\}$, we have
\begin{equation}\label{eq:tangentefibras}
dr_i(X_j)=X_j(r_i)=X_j(\rho_i\circ p_i\circ \psi_i)=0,
\end{equation}
because $\rho_i\circ p_i\circ \psi_i$ is constant along the leaves of $\fH$.
On the other hand, for all $i,j\in I$, we have that $\phi_i ^* \overline{\chi}=\phi_j ^* \overline{\chi}$,
because $\phi_j\circ\phi_i^{-1}\in G$ and $\overline{\chi}$ is $G$-invariant. Thus, fixing $i_0\in I$,  we have
\begin{equation}\label{eq:fijamosi0}
\sum_{i\in I} dr_i(Y)\phi_i^*\overline{\chi}(X_{1},\dots,X_{p})
=
\phi_{i_0}^*\overline{\chi}(X_{1},\dots,X_{p})\sum_{i\in I} dr_i(Y)\\
=0,
\end{equation}
because $\sum\limits_{i\in I} dr_i(Y)=d\bigl(\sum\limits_{i\in I} r_i\bigr)(Y)= d(1)(Y)=0$. Now, \eqref{eq:tangentefibras} and \eqref{eq:fijamosi0} yield
$$
\sum_{i\in I} dr_i \wedge \phi_i^*\overline{\chi}(X_1,\dots,X_p,Y)=0,
$$
which, finally, leads to
\begin{equation*}
d\chi(X_1,\dots,X_p,Y)= \sum_{i\in I} r_i\cdot d\overline{\chi}(\phi_{i*}X_1,\dots,\phi_{i*}X_p,\phi_{i*}Y)=0,
\end{equation*}
because $\overline{\chi}$ is $d\fF$-closed and $\phi_{i*}X_j\in C^{\infty}(T\fF)$ for all $i\in I$ and $j=1,\dots,p$. 
\end{proof}

\begin{proposition}\label{prop:thickasabrick}
Let $\pi\colon (E,\fE)\longrightarrow (B,\fB)$ be a thick foliated bundle of CERFs with taut fibre and compact structure group. Then  the Álvarez classes of $\fE$ and $\fB$ satisfy 
$$
\kappab_{\fE}=\pi^*\kappab_{\fB}.
$$
\end{proposition}

\begin{proof}
The associated fibrewise foliation $\fH$ is taut by Lemma~\ref{lema:particiones de la unidad}. Take $\gamma$ a taut metric on $(E,\fH)$, that is, satisfying $\kappa_{\gamma}=0$, and take $\mu_B$ a strongly tense metric on $(B,\fB)$. By Lemma~\ref{lem:thickasabrick}, we get a metric $\mu_E$ such that $\kappa_{\mu_E}=\pi^*\kappa_{\mu_B}$. Notice that $\kappa_{\mu_E}$ is a closed $\fE$-basic form, and thus, $\mu_E$ is strongly tense. By \ref{subsec:tautCERF}\eqref{item:kappaCERF}, we get $\kappab_{\fE}=\pi^*\kappab_{\fB}$.
\end{proof}

We finish this section with two illustrations of Proposition~\ref{prop:thickasabrick}.
\begin{remark}
Denote by $\pi\colon (M^1,\F^1)\longrightarrow(M,\F)$ the  transverse orthonormal frame bundle of the RF $(M,\F)$. Then  $\pi$  is a thick foliated principal bundle. The fibre $O(q)$ carries the pointwise foliation, which is trivially taut. Then  by Proposition~\ref{prop:thickasabrick} we get that their \'Alvarez classes satisfy $\kappab_{\F^1}=\pi^*\kappab_{\F}$. This result was proven in \cite[Lemma~7]{nozawa-manuscripta} without using Dom{\'\i}nguez's Theorem. This remark applies for any other foliated principal bundle with compact structure group.
\end{remark}

As an application, if we want to study the tautness of the foliation induced by the action of the product of some groups on a manifold, we can drop the compact factors in the sense precised by the following result.

\begin{corollary}\label{cor:actions}
Let $G_1$ be a compact Lie group and $G_2$ be another Lie group, with $G_1\times G_2$ acting freely on $M$. Denote by $\F$ and $\F_2$ the foliations induced on $M$ by $G_1\times G_2$ and $G_2$, respectively, and by $\widetilde{\F}_2$ the foliation induced by $G_2$ on $M/G_1$. Suppose that those foliations are also CERFs. Then, the following statements are equivalent:
\begin{enumerate}[{\rm (i)}]
	\item $(M,\F)$ is taut;
	\item  $(M,\F_2)$ is taut;
	\item $(M/G_1,\widetilde{\F}_2)$ is taut.
\end{enumerate}
\end{corollary}
\begin{proof}

Consider $\F'=\F$ or $\F_2$.
Then the principal bundle $\pi\colon (M,\F')\longrightarrow (M/G_1,\widetilde{\F}_2)$ is a thick foliated bundle whose compact fibre $G_1$ carries either the one leaf foliation or the pointwise foliation, respectively. As both foliations on $G_1$ are taut, the hypothesis of Proposition~\ref{prop:thickasabrick} are satisfied and the equivalence holds.\qedhere
\end{proof}

\section{Local structure of SRFs}\label{sec:local structure}

Let $\K$ be an SRF on the compact manifold $X$. In this section we prove that the \'Alvarez class of a stratum of $\K$ induces the \'Alvarez class of its corresponding sphere bundle, which will be a key step to extend the \'Alvarez class to the whole manifold $X$ in the next section.

\medskip

We first show that the regular part of the sphere bundle of a singular stratum is a CERF. Recall the quasi-isomorphism notation used in \ref{subsec:CERF} (c).

\begin{lemma}\label{lem:tariscos}
	Let $(Y,\mathcal{K})$ be an SRF with a foliated Thom--Mather system. Consider $S$ and $S'$ two singular strata and take $\varepsilon\in(0,\small{\frac{1}{2}})$. Consider $M=D_S\cap R$ the regular part of the sphere bundle $D_S$,
	 and the subset $M_{\varepsilon}=M\backslash \rho_{S'}^{-1}([0,\varepsilon])$. Then  we have $M_{\varepsilon}\qi M$. 
\end{lemma}

\begin{proof}
	First notice that if $S=S'$, then $M\subset D_S\subset \rho_{S}^{-1}(\small{\frac{1}{2}})$, and  hence, $M=M_{\epsilon}$. If $S$ and $S'$ are not comparable, then by Remark~\ref{rem:notrelated} we also have $M=M_{\epsilon}$. So  in both cases, the lemma follows trivially. 
	
	Suppose that either $S\prec S'$ or $S'\prec S$ holds. The restriction of the map  \eqref{eq:explosion} \begin{equation*}
	{\mathscr{L}}_{S'}\colon ({D}_{S'}\times (0,1),\K \times \I) \to (({T}_{S'}\backslash S'),\K),
	\end{equation*}
	to the regular part is the foliated diffeomorphism:
	\begin{equation}\label{eq:exploenlematec}
	{\mathscr{L}}_{S'}\colon ({D}_{S'}\cap R\times (0,1),\K \times \I) \longrightarrow ((T_{S'}\cap R),\K).
	\end{equation}
	We now prove the following identity by double inclusion:  
	\begin{equation}\label{eq:claimexplo}
	{\mathscr{L}}_{S'}(D_S\cap D_{S'}\cap R\times(0,1))=D_S\cap T_{S'}\cap R.
	\end{equation}
	For the  ``$\subseteq$'' part, take $x\in D_S\cap D_{S'}\cap R$ and $t\in(0,1)$. Then  
	\begin{equation}\label{eq:uso1deTM4}
	\rho_S(\mathscr{L}_{S'}(x,t))=\rho_S(H_{S'}(x,2t))\stackrel{(TM4)}{=}\rho_S(x)=1/2,
	\end{equation}
	which implies $\mathscr{L}_{S'}(x,t)\in D_S$, and by \eqref{eq:exploenlematec}, we are done. For the reciprocal, take $x\in D_S\cap T_{S'}\cap R$. As $x\in T_{S'}$, there exists $(y,t)\in D_{S'}\times(0,1)$ such that 
	\begin{equation}\label{eq:uso2deTM4}
	x=H_{S'}(y,2t)\Rightarrow 1/2=\rho_S(x)=\rho_S(H_{S'}(y,2t))\stackrel{(TM4)}{=}\rho_S(y),
	\end{equation}
	that is, $y\in D_S$, which proves the ``$\supseteq$'' part.
	
	We now define the covering $M=M_{\varepsilon}\cup V$, where $V=M\cap T_{S'}$. We have the following chain of foliated diffeomorphisms:
	\begin{equation*}
	\begin{split}
	M_{\varepsilon}\cap V 
	&= D_S\cap R\cap T_{S'}\cap\rho_{S'}^{-1}((\varepsilon,1))
	\stackrel{\eqref{eq:claimexplo}}{=} 
	\mathscr{L}_{S'}(D_S\cap D_{S'}\cap R\times(\varepsilon,1))\\
	&\cong D_S\cap D_{S'}\cap R\times(\varepsilon,1) 
	\cong D_S\cap D_{S'}\cap R\times(0,1)
	\stackrel{\eqref{eq:claimexplo}}{\cong} D_S\cap R\cap T_{S'}=V,
	\end{split}
	\end{equation*}
	and thus the Mayer-Vietoris sequence for basic cohomology yields $M_{\varepsilon}\qi M$.
\end{proof}
\begin{remark}
	Notice that we need both parts of (TM4) to prove \eqref{eq:uso1deTM4} and \eqref{eq:uso2deTM4} because we are considering both cases ($S\prec S'$ and $S'\prec S$).
\end{remark}

We now fix some notation for the rest of this article. For each $i \in \Z$  we shall write:
\begin{itemize}
    \item $\Sigma_{i} =\displaystyle{\bigcup_{\mathclap{\depth_{\K} S\leq i }}}S$
\end{itemize}
and, for $i \in  \{0, \ldots , r-1 \}$,  where $r= \depth \SK$,
\begin{itemize}
\item $T_i=\displaystyle{\bigcup_{\mathclap{S\subset \Sigma_{i}\backslash \Sigma_{i-1} }}}T_S$
\item $\tau_{i} \colon T_{i} \rightarrow  \Sigma_{i}\backslash \Sigma_{i-1}$;
\item $\rho_{i} \colon T_{i}\to [0,1)$ its radius function, and
\item $D_{i} = \rho_i^{-1}(\small{\frac{1}{2}})$ the core of $T_{i}$.	
\end{itemize}

Notice that both $T_i$ and $D_i$ have a finite number of connected components. The proof of the next proposition is similar to that of \cite[Proposition 2.4]{tope2}:

\begin{proposition}\label{prop:TSR}
	Let $S$ be a singular stratum of $\K$, denote by $T_S$ its tubular neighbourhood in $\mathcal{T}$ and $D_S$ its corresponding core. Then  the restrictions of $\K$ to both $T_S\cap R$ and $D_S\cap R$ are CERFs.
\end{proposition}

\begin{proof}
	The foliated diffeomorphism \eqref{eq:exploenlematec} implies that it suffices to show that $( D_S \cap R,\K)$ is a CERF to prove the proposition. Notice that any zipper of the CERF $(R,\K)$ is a zipper for $( D_S \cap R,\K)$, which yields property (a) of \ref{subsec:CERF}. So  it suffices to construct a reppiz of $M=D_S\cap R$ by removing a small neighbourhood of each singular stratum.
	
	Take $\varepsilon\in(0,\small{\frac{1}{2}})$ and consider $T_i^{\varepsilon}={\displaystyle \bigcup_{j=0}^{i}\rho_{j}^{-1}([0,\varepsilon])}$ for all $i=0,\dots,r-1$.
	Notice that $X\backslash R=\Sigma_{r-1}\subset T_{r-1}^{\varepsilon}$. We  now show that $U=M\backslash T_{r-1}^{\varepsilon}$ 
	is a reppiz of  $M$, i.e., $U$ satisfies properties (b) and (c) of \ref{subsec:CERF}. 
	\medskip

	We define, for $i=0,\dots,r-1$ :
	\begin{itemize}
		\item the subset $U_i=M\backslash T_i^{\varepsilon}$;
		\item the space $Y_i=X\backslash T_i^{\varepsilon}$;
		\item the collection $\mathcal{T}_i=\Set{ T\backslash T_i^{\varepsilon} | T\in\mathcal{T} }$;
		\item the statement\newline 
		$
		P_i\equiv$ ``$\mathcal{T}_i$ is a foliated Thom--Mather system of the SRF $(Y_i,\mathcal{K})$,  and  $U_i\qi M$.''
	
	\end{itemize}	
	We now prove $P_{r-1}$ by induction on $i$, which implies \ref{subsec:CERF} (b).

	As $\Sigma_0$ is a union of minimal (closed) strata of $(X,\mathcal{K})$, then $T_0^{\varepsilon}=\rho_0^{-1}([0,\varepsilon])$ is a saturated closed subset of $(X,\mathcal{K})$, and the first part of $P_0$ follows. Now by repeatedly applying Lemma~\ref{lem:tariscos} with $Y=X$ and $S'$ each connected component of $\Sigma_0$, we get the second part, that is, $U_0\qi M$.
	
	Suppose now that $P_{i-1}$ is true, and let's prove $P_i$. We have 
	$$
	Y_i
	=X\backslash T_i^{\varepsilon}
	=Y_{i-1}\backslash\rho_i^{-1}([0,\varepsilon])
	\quad \text{and} \quad 
	\mathcal{T}_i
	=\Set{T\backslash \rho_i^{-1}([0,\varepsilon]) | T\in\mathcal{T}_{i-1}}.
	$$ 
	Notice that, for every $S'\in\Sigma_i$ we have that $S'\backslash T_{i-1}^{\varepsilon}$ is a minimal (hence, closed) stratum of $Y_{i-1}$. So  $\rho_i^{-1}([0,\varepsilon])$ is a saturated closed subset of $(Y_{i-1},\mathcal{K})$ and as a consequence, we get the first part of $P_i$. By repeatedly applying Lemma~\ref{lem:tariscos} with $Y=Y_{i-1}$ and $S'$ each connected component of $\Sigma_i$, we get $U_i\qi U_{i-1}$. By induction hypothesis, we have $U_{i-1}\qi M$, and hence, $U_i\qi M$ holds. So  we get $P_i$ and the induction proof is completed. 
	
	We have proven $P_{r-1}$, which implies that $U=U_{r-1}$ satisfies \ref{subsec:CERF} (b).
	
	\medskip
	
	We now have to prove that the closure of $U$ in $M$ is compact (property (c) of \ref{subsec:CERF}). 
	
	Consider the union of open tubes $T_{\varepsilon}=
	\displaystyle{\bigcup_{i=0}^{r-1} (T_i^{\varepsilon})^{\circ}}=
	{\displaystyle \bigcup_{i=0}^{r-1}\rho_{i}^{-1}([0,\varepsilon))}$, where $(\cdot)^{\circ}$ stands for the interior operator. We define $K=M\backslash T_{\varepsilon}$, which is a subset of $M$ containing $U$. Let's compute its closure in $X$: 
	$$
	\overline{K}
	\subset \overline{M}\backslash (T_{\varepsilon})^{\circ} 
	= \overline{M}\backslash T_{\varepsilon}
	\stackrel{(\star)}{=} M\backslash T_{\varepsilon} = K.
	$$
	So  $K$ is closed in $X$, and thus, compact. Hence, $U$ is compact, and \ref{subsec:CERF} (b) is satisfied. To justify step $(\star)$, it suffices to see  $\overline{(D_S\cap R)}\backslash (D_S\cap R) \subset \Sigma_{r-1}$, because in that case,
	$\overline{M}\backslash M  \subset T_{\varepsilon}.$
	By construction, the core $D_S$ is a closed subset of
	$X \backslash \Sigma_{s-1}$, being $s=\depth_{\K} S$. So  we get $\overline{D_S}\backslash
	D_S
	\subset \Sigma_{s-1}$ and therefore:
	$$
	\overline{D_S \cap R}\backslash (  D_S \cap R)
	\subset
	\overline{ D_S} \backslash (  D_S \cap R)
	=
	\overline{D_S}  \backslash D_S
	\cup
	\overline{D_S}  \backslash R
	 \subset \Sigma_{s-1} \cup X\backslash R 
	 =\Sigma_{r-1},
	$$
	which ends the proof.\qedhere
	\end{proof}

We get that the \'Alvarez class of a singular stratum induces that of its sphere bundle.

\begin{proposition}\label{prop:depth1}
Let $S$ be a minimal stratum of an SRF $(X,\mathcal{K})$, and let $\pi\colon D_S\longrightarrow S$ be its associated sphere bundle. Suppose that the foliation of the fibre is taut.
Then  the \'Alvarez classes of $\mathcal{K}_{D_S\cap R}$ and $\mathcal{K}_{S}$ satisfy $\kappab_{D_S\cap R}=\pi^*\kappab_S$.
\end{proposition}

\begin{proof}
From Proposition~\ref{prop:TSR}, $\mathcal{K}_{D_S\cap R}$ is a CERF. 
From Lemma~\ref{lema:struct}, $\pi$ is a thick foliated bundle with compact structure group and taut fibre.
Proposition~\ref{prop:thickasabrick} gives the result.
\end{proof}

\section{Tautness of Singular Riemannian Foliations}

Let $\K$ be an SRF on the compact connected manifold $X$. In Proposition \ref{prop:depth1} we proved that the Álvarez class of each stratum is related to that of the regular part of a tube around it. This resembles a lot Verona's approach to differential forms on stratified spaces, cf. \cite{verona}.
In this section we shall follow that approach to patch up the Álvarez classes of all the strata into a unique cohomology class.

Notice that the definition of the basic forms (see Section \ref{subsec:CERF}) in the context of a regular foliation makes sense when the foliation is singular. We shall thus use the same notation and denote the {\em basic cohomology} of $(X,\K)$ by $H^*(X/\K)$. The existence of basic partitions of unity  implies the existence of Mayer-Vietoris sequences for the basic cohomology of open $\K$-saturated subsets of $X$ (see \cite[Lemma 3]{wolak-subalgebras}).

\begin{remark}\label{remark:inyective}
By degree reasons, the inclusion $\Omega^1(X/\K)\hookrightarrow \Omega^1(X)$ induces a monomorphism in cohomology
$H^1(X/\K)\rightarrowtail H^1(X)$.
\end{remark}

The following example will be used in the proof of the main theorem:

\begin{lemma}\label{lema:esfera}
Let $\G$ be an SRF on the sphere $\S^{k}$ without 0-dimensional leaves. Then  $H^1(\S^{k}/\G)=0$.
\end{lemma}
\begin{proof}
If $k=1$, since there are no 0-dimensional leaves, $\G$ must be a regular foliation of dimension 1, that is, the one leaf foliation, and the statement holds trivially. If $k\ne 1$, then the lemma follows by Remark~\ref{remark:inyective}, because  $H^1(\S^k)=0$.
\end{proof}

Assume the notation used in Section \ref{sec:local structure}, and put $r=\depth\SK$. We have the following compact saturated subsets of $X$:
$$
\emptyset=\Sigma_{-1}\subset\Sigma_0\subset\cdots\subset\Sigma_{r-1}=X\backslash R\subset\Sigma_r=X.
$$
We also consider the open saturated subsets $R_i=X\backslash\Sigma_{r-i}$ and the inclusions
$$
\emptyset=R_0\subset R=R_1\subset R_2\subset\cdots\subset R_r\subset R_{r+1}=X.
$$ 
\begin{proposition}\label{prop:mono-regular}
Let $\K$ be an SRF on $X$, and denote by $R$ its regular part. Then the inclusion  induces a monomorphism in basic cohomology 
	$$\iota_R\colon H^1(X/\K)\rightarrowtail H^1(R/\K).$$
\end{proposition}
\begin{proof}
We consider, for $i\in\{1,2,\dots,r+1 \}$ the following statement:
$$
P_i\equiv\text{``The inclusion } R_{i-1}\hookrightarrow R_i \text{ induces a monomorphism } H^1(R_{i}/\K)\rightarrowtail H^1(R_{i-1}/\K) \text{''}.
$$
We shall show that $P_{i}$ holds for every  $i\in\{1,2,\dots,r+1 \}$, thus proving the Proposition. $P_1$ holds trivially. Take, for short, $T=T_{r-i+1}$, which is a tube of the singular strata forming $R_i\backslash R_{i-1}$. From the saturated open covering $\{R_{i-1},T\}$ of $R_i$, we have the Mayer-Vietoris exact sequence for basic cohomology, which begins:
\begin{equation}\label{eq:beginningsequence}
H^0(R_i/\K) \longrightarrow H^0(R_{i-1}/\K)\oplus H^0(T/\K)\longrightarrow H^0((R_{i-1}\cap T)/\K)\longrightarrow\dots 
\end{equation}
By \eqref{eq:beginningsequence},  $H^0(R_i/\K)\cong H^0(R_{i-1}/\K)$ and $H^0(T/\K)\cong H^0((R_{i-1}\cap T)/\K)$, we get the exact sequence:
\begin{equation}\label{eq:propiny-mv}
0\longrightarrow H^1(R_i/\K) \stackrel{\phi}{\longrightarrow} H^1(R_{i-1}/\K)\oplus H^1(T/\K)\stackrel{\rho}{\longrightarrow} H^1((R_{i-1}\cap T)/\K).
\end{equation}
If we prove that $H^1(T/\K){\longrightarrow} H^1((R_{i-1}\cap T)/\K)$ is injective, we get that $\phi$ is injective and thus $P_i$. As $R_{i-1}\cap T=T_{r-i+1}\backslash\Sigma_{r-i+1}$, it suffices to prove that, for each singular stratum $S$, the inclusion $\iota\colon T_S\cap R_{i-1}\hookrightarrow T_S$ induces a monomorphism in basic cohomology. We have
\begin{equation}\label{eq:tubito-morro}
\xymatrix{
	H^1(T_S/\K)\ar[r]^-{\iota^*} & H^1((T_S\cap R_{i-1})/\K)\ar[d]^{\cong}_{\iota_D^*}\\
	H ^1(S/\K)\ar[u]^{\tau^*}_{\cong}\ar[r]^-{\pi*}   & H^1(D_S\cap R_{i-1})/\K),
}
\end{equation}
where $\iota_D$ is induced by the inclusion $D_S\hookrightarrow T_S$ and $\pi=\tau|_{D_S}$. As $\iota_D^*$ is an isomorphism we just have to prove that $\pi^*$ is injective. Take $[\alpha]\in H^1(S/\K)$ so that $\pi^*\alpha=df$, with $f$ a basic function on $D_S\cap R_{i-1}$. As $i_Xdf=i_X(\pi^*\alpha)=0=i_Xf$ for every $X\in\ker\pi_*$, we have that $f$ is $\pi$-basic, and so there exists $g\in C^{\infty}(S)$ so that $f=\pi^*g$. Hence $\pi^*(\alpha)=d\pi^*g=\pi^* dg$, which yields $\alpha=dg$, because $\pi$ is a submersion. By degree reasons, we have that $g$ is $\K$-basic and thus $[\alpha]=[dg]=0$ in $H^1(S/\K)$. We get that $\pi^*$ is injective, which ends the proof.
\end{proof}

\setcounter{section}{1}
\setcounter{theorem}{0}
\begin{theorem}[{\bf bis}.]
	Let $\K$ be an SRF on a closed manifold $X$. Then  there exists a unique class $\kappab_X\in H^1(X/\K)$ that contains the \'Alvarez class of each stratum. More precisely, the restriction of $\kappab_X$ to each stratum $S$ is the \'Alvarez class of $(S,\K_S)$.
\end{theorem}
\setcounter{section}{6}
\setcounter{theorem}{3}
\begin{proof}
First notice that the unicity of $\kappab_X$ comes from the fact that it induces the Álvarez class of the regular part and from Proposition \ref{prop:mono-regular}. For the existence, 
we proceed by complete induction on $r=\depth\SK$. When $r=0$ the SRF is indeed an RF and the result follows trivially. If $r>0$, we assume that the theorem is true for $\depth\SK<r$ and prove it for $\depth\SK=r$. We shall construct inductively on $i\in\{0,1,\dots,r+1\}$ a form $\kappab_i\in H^1(R_i/\K)$ satisfying 
$$
\kappab_i\vert_{R_{i-1}}=\kappab_{i-1}\quad  \text{ and }\quad  \kappab_i\vert_S=\kappab_S\quad \forall S\subset R_i\backslash R_{i-1},
$$ 
and finish the proof by taking $\kappab_X=\kappab_{r+1}$. Suppose that we have constructed the classes $0=\kappab_0,\kappab_1,\dots,\kappab_{i-1}$  and let's construct $\kappab_i$. As in the proof of Proposition \ref{prop:mono-regular} we take the open covering $\{R_{i-1},T\}$ of $R_i$,  which  yields the exact sequence \eqref{eq:propiny-mv}. We shall prove that $\rho(\kappab_{i-1},\tau^*\kappab_{S})=0$ for each singular stratum $S\subset R_i\backslash R_{i-1}$, which would give $\kappab_i$ by the exactness of \ref{prop:mono-regular}, thus completing the induction.
By Proposition \ref{prop:mono-regular}, $\iota_i\colon H^1((R_{i-1}\cap T)/\K)\rightarrowtail H^1((T\cap R)/\K)$ is a monomorphism. So, if we prove that $\iota_i\circ\rho(\kappab_{i-1},\tau^*\kappab_{S})=0$, we are done. We have
\begin{equation}\label{eq:hayqueanular}
\iota_i\circ\rho(\kappab_{i-1},\tau^*\kappa_{S})=\kappab_{i-1}|_{R\cap T}-\tau^*\kappab_S|_{R\cap T}=\cdots=\kappab_{R}|_{R\cap T}-\tau^*\kappab_S|_{R\cap T},
\end{equation}
where $\kappab_R$ is the Álvarez class of $R=R_1$. Notice that the nullity of \eqref{eq:hayqueanular} can be checked on $H^1(D_S\cap R)/\K)$ via the isomorphism $\iota_D^*$ of \eqref{eq:tubito-morro}. There only remains to prove that
\begin{equation}\label{eq:anular-sobre-esfera}
\kappab_{R}\vert_{R\cap D_S} =\pi^* \kappab_S\in H^1((D_S\cap R)/\K).
\end{equation}
Let $(\S^{n_S},\G)$ be the fibre of the bundle $D_S$. Then $\depth\G<r$, and by induction hypothesis, there exists a class $\kappab_{\S}\in H^1(\S^{n_S}/\G)$ whose restriction to the regular part $R_{\S}$ of $\G$ is the Álvarez class of $(R_{\S},\G)$. By Lemma  \ref{lema:esfera}, we have that $\kappab_{\S}=0$, which yields $\kappab_{R_{\S}}=0$ and thus $(R_{\S},\G)$ is a taut CERF. Then by Proposition \ref{prop:depth1} we get \eqref{eq:anular-sobre-esfera} and the proof is complete.
\end{proof}

This theorem leads us to the following natural definitions:

\begin{definition}\label{def:tautnessclass-SRF}
Let $\K$ be an SRF on a compact connected manifold $X$. Then the {\em Álvarez class} of $(X,\K)$ is the unique class $\kappab_X\in H^1(X/\K)$ which induces the Álvarez class of every stratum of $\SK$. We shall say that an SRF is {\em cohomologically taut} if its Álvarez class is zero.
\end{definition}

Notice that although geometrical tautness cannot be achieved globally for an SRF (see \ref{subsec:tautnesscohomology}), we have been able to define cohomological tautness by means of a basic class (which may be regarded as a class in $H^1(X)$, by Remark \ref{remark:inyective}). The geometrical meaning of the cohomological tautness of an SRF must be interpretated individually on each stratum, as we summarize in the following theorem:

\begin{theorem}\label{th:carac}
Let $\K$ be an SRF on a compact manifold $X$. Then, the  following three statements are equivalent:
\begin{enumerate}[{\rm(a)}]
	\item The foliation $\K$ is cohomologically taut;
	\item The foliation $\K_S$ is taut for each stratum $S\in\SK$;
	\item The foliation $\K_R$ is taut.
\end{enumerate} 
\end{theorem}
\begin{proof}
The only nontrivial implication is $({\rm c})\Longrightarrow ({\rm a})$, which follows because $\iota_R(\kappa_X)=\kappab_R=0$ and Proposition \ref{prop:mono-regular}.
\end{proof}

The following Corollary generalizes E.~Ghys' celebrated result about tautness of Riemannian foliations on simply connected spaces (see \cite[Th\'eor\`eme B]{ghys}:

\setcounter{section}{1}
\setcounter{theorem}{1}
\begin{corollary}[{\bf bis}.]\label{cor:gureghys}
Every SRF on a compact simply connected manifold $X$ is cohomologically taut.	
\end{corollary}
\setcounter{section}{6}
\setcounter{theorem}{5}
\begin{proof}
By Remark \ref{remark:inyective} we have $H^1(X/\K)=0$, and thus $\kappab_X=0$.
\end{proof}

\begin{remark}
Recall that Molino's desingularization $(\widetilde{X},\widetilde{\K})$ of $(X,\K)$ is an RF that is taut if and only if $\K_{R}$ is taut \cite[Remark 2.4.3]{tope1}. Nevertheless, Corollary \ref{cor:gureghys} cannot be proved directly from that fact, because the desingularization of a simply connected manifold may not be simply connected. Notice also, that, as a consequence of Theorem \ref{th:carac} and Corollary \ref{cor:gureghys}, the foliation induced on each stratum of a SRF on a simply connected compact manifold is a geometrically taut RF, which is not evident a priori.
\end{remark}

As in the regular case, cohomological tautness can be detected by using some other cohomological groups.  The first one is the twisted cohomology ${H}^{0}_{\kappa }(X/\K)$ where $\kappa$ is any representative of the Álvarez class ${\kappab}_{X}$.

\begin{proposition}
	Let $X$ be a connected compact manifold endowed with an SRF $\K$. The  following two statements are equivalent:
	\begin{enumerate}[{\rm(a)}]
		\item The foliation $\K$ is  cohomologically taut.
		\item The cohomology group ${H}^{0}_{\kappa}(X/\K)$ is isomorphic to $\R$.
	\end{enumerate}
Otherwise, ${H}^{0}_{\kappa}(X/\K)=0.$
\end{proposition}
\begin{proof}

We proceed in two steps.

${\rm(a)} \Rightarrow {\rm(b)}$. If $\K$ is cohomologically taut then $\kappab_{X} =
[\kappa]=0$. So,
${H}^{0}_{\kappa }(X/\K) \cong {H}^{0}(X/\K) = \R$.

${\rm(b)} \Rightarrow {\rm(a)}$. From Theorem \ref{th:carac} it suffices to prove that ${\K}_{R}$ is a taut foliation; that is, that
${H}^{0}_{\kappa_R}(R/\K)  \ne 0$ (cf. \cite[Theorem 3.5]{tope2}). Proceeding as in Proposition \ref{prop:mono-regular}, we get that the restriction $ \iota_R \colon {H}^{0}_{\kappa }(X/\K)  \to {H}^{0}_{\kappa_R}(R/\K)$ is a monomorphism. This gives (a).

Since ${H}^{0}_{\kappa_R}(R/\K) = 0 \hbox{ or } \R$ (cf. \cite[Theorem 3.5]{tope2}) then we get  $ {H}^{0}_{\kappa }(X/\K) =  0 \hbox{ or } \R $.
\end{proof}

The second one has been proved in \cite[Corollary 3.5]{topinthebic}. Recall that a singular stratum is a {\em boundary stratum\footnote{In \cite[p.~431]{topinthebic} there's a typo in one sign of the formula of the definition, which says  $\codim_M\F = \codim_{S'}\F_{S'} -1$.}} if there exists a stratum $S'$ satisfying $S\preccurlyeq   S'$ and  $\codim_X\K = \codim_{S'}\K_{S'} +1$. The union of boundary strata of $\K$ is denoted by $\partial(X/\K)$.

\begin{proposition}\label{prop:topindebic}
Let $X$ be a connected compact manifold endowed with a CERF $\K$ such that
 ${\K}_{R}$ is
transversally oriented.
Put $n = \codim {\K}_{R}$.
Then, the following three statements are equivalent:

\begin{enumerate}[{\rm(a)}]
	\item The foliation $\K$ is  cohomologically taut.
	\item The cohomology group ${H}^{n}(X/\K, \partial (X/\K))$ is isomorphic to $\R $.
	\item  The intersection cohomology group ${I\!H}^{n}_{\overline{p}}(X/\K)$ is isomorphic to $\R $,
	for any perversity $\overline{p} \leq \overline{t}$.
\end{enumerate}

\end{proposition}

The inductive construction of the Álvarez classes $\kappab_i$ in the proof of Theorem \ref{theorem:main} implies that if $S'\preccurlyeq S$ and $\K_S$ is cohomologically taut, then $\K_{S'}$ must be cohomologically taut. In fact, the \'Alvarez class is an obstruction to foliatedly embed Riemannian foliations: it is not possible to foliatedly embed a non-taut RF in a taut SRF. As an application of Corollary \ref{cor:gureghys}, we can, for example, get that Carri\`ere's well-known non-taut Riemannian flow on the manifold $T^3_A$ (see \cite[Exemple I.~ D.~ 6]{carriere}) cannot be foliatedly embedded in any sphere with a SRF.

Nevertheless, both taut and non-taut strata can coexist in the same SRF, as the following example shows. 

\begin{example}(\cite[Section 3.5]{royotesis}) \label{example:surgery}
	Consider the unimodular matrix $A=\left(\begin{smallmatrix}2&1\\1&1\end{smallmatrix}\right)$ and $v=(v_1,v_2)$ one of its irrational slope eigenvectors. Denote by $\F_{v}$ the Kronecker flow induced on the torus $\T^2$, which can be naturally extended to the linear flow $\F_{\S^3}$ on $\S^3$ corresponding to the $\R$-action $\Phi_{\S^3}(t,(z_1,z_2))=(e^{i\cdot v_1t}\cdot z_1,e^{i\cdot v_2t}\cdot z_2)$. We consider the suspension of this action, that is, the $\R$-action on $\S^4=\Sigma\S^3$, given by $\Phi(t,[(z_1,z_2),s])=[\Phi(t,(z_1,z_2)),s]$. Its orbits define an SRF $\F_{\S^4}$ whose singular part consists of two fixed points: the North and South poles of $\S^4$. By Corollary \ref{cor:gureghys}, $\F_{\S^4}$ is a cohomologically taut singular Riemannian flow.

	The closures of the generic leaves of $\F_{\S^4}$ are tori of dimensions 1 and 2. Take two leaves $L_1$ and $L_2$ whose disjoint closures are foliated diffeomorphic to $(\T^2,\F_{v})$, and consider $K_1,K_2$ disjoint compact saturated neighbourhoods of $\overline{L_1}$ and $\overline{L_2}$. Then $(K_i,\F_{\S^4})$ is foliated diffeomorphic to $(\T^2\times \D^2,\F_v\times{\text{\{points\}}})$  for $i=1,2$ (see \cite[Proposition 3]{carriere}). Notice that $(\partial{K_i},\F_{\S^4})\cong(\T^3,\F_v\times{\text{\{points\}}})$, and denote by  $(K_i)^{\circ}$ the interior of $K_i$ for $i=1,2$. We construct the manifold
	$X=(\S^4\backslash((K_1)^{\circ}\cup(K_2)^{\circ}))/\sim$, where we have identified $\partial{K_1}$ and $\partial{K_2}$ by $(x_1,-)\sim (A\cdot x_2,-)$ for $x_1,x_2\in\T^2$. Equivalently, we can glue the boundaries of a handle $\T^3\times [0,1]$ to $\partial{K_i}$ using the identification $\sim$ and the identity, respectively, for $i=1,2$ (see Figure \ref{fig:ejemplo}). 
	
		\begin{figure}[t]
			\includegraphics{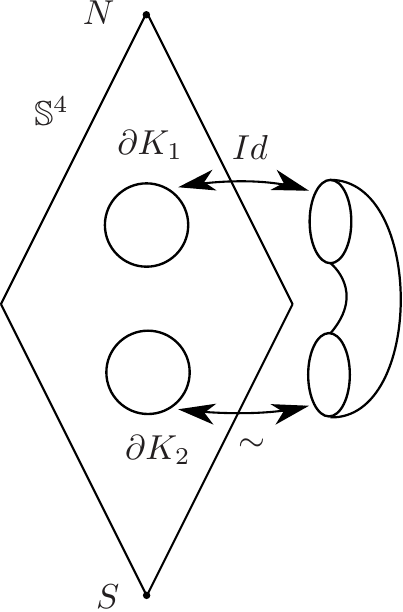}
			\caption{The ends of the handle are diffeomorphic to $\T^3$}
			\label{fig:ejemplo}
		\end{figure}

	 Notice that the described surgery is compatible with $\F_{\S^4}$ and denote by $\K$ the singular Riemannian flow induced in $X$. By construction, $\K$ has two singular strata, both of them fixed points, and thus, trivially endowed with taut foliations. None of them are boundary strata, and then $H^3(X/\K,\partial(X/\K))=H^3(X/\K)$. A straightforward Mayer-Vietoris computation gives $H^3(X/\K)=0$, and by (b) of Proposition \ref{prop:topindebic} we have that $\K$ is not cohomologically taut. 
\end{example}

The described surgery can be easily generalized to generate cohomologically non-taut singular Riemannian flows from cohomologically taut singular Riemannian flows.

\section{Appendix}

This appendix is devoted to proving Lemma \ref{lema:struct}.

\subsection{Some geometrical facts}

The following two lemmas will be useful to describe the local foliated structure of the charts of a tubular neighbourhood of a stratum.

\begin{lemma}\label{lem:appendix:embed}
Let $M$ and $M'$ be two manifolds, respectively endowed with
the SRFs $\mathcal{F}$ and $\mathcal{F}'$.
Consider     an embedding $F \colon M \times[0,1) \to M' \times
[0,1)$. If the restriction
\begin{equation}
\label{eq:hyp}
F\colon (M \times (0,1) , \mathcal{F} \times \mathcal{I})
\to (M' \times (0,1) , \mathcal{F}' \times \mathcal{I})
\end{equation}
is foliated then
\begin{equation*}
F \colon (M \times[0,1) , \mathcal{F} \times \mathcal{I})
\to (M' \times[0,1) , \mathcal{F}' \times \mathcal{I})
\end{equation*}
is also foliated.
\end{lemma}
\begin{proof}
Notice, on one hand, that when $\F$ and $\F'$ are regular, the result
follows directly from the local description of $F$.

Consider,  on the other hand, $S\in \SF$ a minimal stratum. From
\eqref{eq:hyp}
there exists $S'
\in \SFp$ with $F(S \times (0,1)) \subset S' \times (0,1)$ and
therefore
$F(S \times[0,1)) \subset \overline{S'} \times[0,1)$. We claim that
$$F(S \times[0,1)) \subset S' \times[0,1).$$ For that purpose, let
us suppose that
there exists $S'_0 \in \SFp$ with $S'_0 \prec S'$ and $F(S \times \{
0 \}) \cap (S'_0 \times \{
0 \}) \not= \emptyset$.
Since $F(M \times[0,1)) $ is an open subset of $M' \times[0,1)$
then $F(M \times (0,1)) \cap (S'_0 \times (0,1) )\not= \emptyset$.
But this is not possible since the map
$
F\colon (M \times (0,1) , \mathcal{F} \times \mathcal{I})
\to (F(M \times (0,1) ), \mathcal{F}' \times \mathcal{I})
$
is a foliated diffeomorphism and $S \times (0,1)$ a minimal stratum
of
${\sf S}_{\mathcal{F} \times \mathcal{I}}$.

We proceed now by induction on $\depth \SF$. If  $\depth \SF=0$,
then
$\mathcal{F}$ is a regular foliation and the above considerations
yield
$F(M \times[0,1)) \subset R' \times[0,1)$, where $R'$ is a regular
stratum of $\SFp$. We get the result since the two foliations are
regular.

Now, if $\depth \SF>0$, denote $S_{min}$
the union of closed strata of $\mathcal{F}$. By induction hypothesis,
the restriction
$$
F \colon ((M - S_{min}) \times[0,1), \mathcal{F} \times \mathcal{I})
\longrightarrow (M'
\times[0,1) , \mathcal{F}' \times \mathcal{I})
$$
is a foliated map. Consider now $S \in \SF
$ a singular stratum. We have seen that there exists $S'\in \SFp$
with
$F(S \times[0,1)) \subset S' \times[0,1)$. It remains to prove that
$$
F\colon (S \times[0,1), \mathcal{F} \times \mathcal{I}) \longrightarrow (S'
\times[0,1) , \mathcal{F}' \times \mathcal{I})
$$
is a foliated map. It follows, since $\F$ and $\F'$ are regular.
\end{proof}

\begin{lemma}
\label{lem:appendix:borde}
Let $M$ and $M'$ be two manifolds, respectively endowed with
the SRFs $\mathcal{F}$ and $\mathcal{F}'$. Consider an embedding
$f\colon M\times\R^{n+1}\to M'\times\R^{n+1}$ with $f(M\times\{0\})\subset M'\times\{0\}$. Then  there exists a unique embedding
$$
F\colon M\times\S^n\times[0,1)\longrightarrow M'\times\S^n\times[0,1)
$$
making the following diagram commutative
$$
\xymatrix{
	M \times  \S^{n }  \times[0,1)\ar[r]^F \ar[d]_Q &  M' \times
	\S^{n }  \times[0,1)\ar[d]_Q\\
	M \times  \R^{n+1}\ar[r]^f & M' \times  \R^{n+1}}
$$
where
the smooth map $Q$ is defined by $Q(u,\theta,t) = (u,t\cdot\theta)$. Moreover, consider SRFs 
$\E$ and $\G$ on $\R^{n+1}$ and $\S^n$, respectively, such that 
$
P \colon \left( \S^{n } \times (0,1) ,{\mathcal{G}}
\times
\mathcal{I} \right)\to \left (\R^{n+1} \backslash \{ 0 \},
\E\right),
$
defined by $P(\theta,t)=t\cdot \theta$, is a foliated diffeomorphism and $\mathcal{I}$ denotes the foliation by points.
Suppose that the embedding
$f \colon (M \times \R^{n+1} , \F \times {\mathcal{E}})
\to (M' \times \R^{n+1} , \F' \times {\mathcal{E}})$ is foliated. Then  
$$
F \colon (M \times  \S^{n } \times[0,1) , \F
\times {\mathcal{G}}  \times \mathcal{I}) \to
(M' \times  \S^{n } \times[0,1) , \F'
\times {\mathcal{G}} \times \mathcal{I})
$$
is a foliated embedding.
\end{lemma}
\begin{proof}
	We proceed in several steps.
	
	\smallskip
	
	$\bullet$ {\em Uniqueness}. It follows from the density of $M \times
	\S^{n }  \times
	(0,1)$ in $M \times
	\S^{n }  \times
	[0,1)$  and from the fact that the restriction
	$Q \colon M \times  \S^{n }  \times
	(0,1)
	\to M \times (\R^{n+1}  \backslash \{ 0 \})$ is a
	diffemorphism.
	
	\smallskip
	
	$\bullet$  {\em Existence}. Write $f=  (f_{0},f_{1})\colon M \times
	\R^{n+1} \to
	M'\times  \R^{n+1}$. The components $f_{0}$ and $f_{1}$ are
	smooth with
	$f_1(-,0) =0$.
	So  the map $h \colon M \times
	\S^{n }  \times[0,1) \to \R^{n_S+1}$ defined by
	$h(u,\theta,t) =
	f_1(u,t \cdot \theta)/t$  is smooth and
	without zeroes. Finally, we define $$F(u,\theta,t) = \left(f_{0}(u,t
	\cdot
	\theta),\frac{h(u,\theta,t)}{|| h(u,\theta,t)||},
	t \cdot || h(u,\theta,t)||\right).$$
	
	\smallskip
	
	$\bullet$  {\em Embedding}. Since $f = (f_{0},f_1)$ is an embedding
	with $f_1(-,0)=0$, then each restriction
	$f_1(u,-) \colon \R^{n+1} \to \R^{n+1}$ is an
	embedding.
	Put $G_u=f_1(u,-)_{*0}$  its differential at 0, which is an isomorphism. By
	construction we have $h(u,\theta,0) = G_u (\theta)/||G_u (\theta)||$.
	So 
	each restriction $F \colon \{ u \} \times \S^{n }  \times \{ 0
	\} \to
	\{ f_{0}(u,0) \} \times \S^{n }  \times \{ 0 \}$ is a
	diffeomorphism.
	
	Now, consider $(u_i,\theta_i)\in M \times
	\S^{n }$ for $i=1,2$ such that $F(u_1,\theta_1,0)=F(u_2,\theta_2,0)$.
	Since $f $  is an embedding  both  $(u_i,\theta_i,0)$  live on the same fibre $\{ u \}
	\times \S^{n }  \times \{ 0\}$. Since $G_u$ is  an
	isomorphism, we	get that $F$ is an embedding.
	
	\smallskip
	
	$\bullet$  {\em Foliated}.  The restriction
	$ F \colon (M \times \S^{n } \times (0,1) , \F \times
	{\mathcal{G}} \times \mathcal{I}) \to
	(M' \times \S^{n } \times (0,1) , \F' \times
	{\mathcal{G}} \times \mathcal{I})
	$
	is a foliated embedding since
	the restriction
	$
	Q \colon (M \times \S^{n } \times (0,1) , \F \times
	{\mathcal{G}} \times \mathcal{I})
	\to
	(M \times \R^{n +1} \backslash \{ 0 \} , \F \times
	{\mathcal{E}} ) $
	is a foliated diffeomorphism.
	Now, it suffices to apply Lemma  \ref{lem:appendix:embed} to $F$.\qedhere
\end{proof}

\subsection{Lifting of charts}

Assume the notation of Sections \ref{subsec:tubular}-\ref{subsection.struct.group}, some of which we recall now.
As the foliation ${\mathcal{E}}_{S}$ of $\R^{n_S+1}$ is
invariant by 
homotheties, there exists a foliation ${\mathcal{G}}_{S}$ on
the
sphere ${\S}^{ n_S  }$ such that the map $P \colon \left( \S^{ n_S  } \times (0,1) ,{\mathcal{G}}_{S}
\times
\mathcal{I} \right)\to \left (\R^{n_S+1} \backslash \{ 0 \},
{\mathcal{E}}_{S}\right),$ 
where $\mathcal{I}$ denotes the
0-dimensional foliation,  is a foliated diffeomorphism. The foliation ${\mathcal{G}}_{S}$ is an SRF (see \cite{molino}).

We put $\mathcal{A}$ and  $\mathcal{B}$  two atlases of the tubular
neighbourhood $\tau_S \colon T_S \to S$ whose
structure groups are, respectively, $O(n_S+1)$ and $\Diff
(\D^{n_S+1},{\mathcal{E}}_{S})$. Recall the foliated smooth map
$
{\mathscr{L}}_{S} \colon (D_S \times[0,1) , \mathcal{K}
\times
\mathcal{I})\to (T_S, \mathcal{K})
$
defined by
${\mathscr{L}}_{S} (z,t) = 2t \cdot z$.
The restriction
$
{\mathscr{L}}_{S} \colon (D_S \times (0,1) , \mathcal{K}
\times
\mathcal{I})\to (T_S\backslash S, \mathcal{K})
$
is a foliated diffeomorphism. Notice that any chart $(U,\psi) \in \mathcal{A}$ satisfies $\psi(v, t \cdot \theta ) = t \cdot \psi(v,\theta)$ and induces a lifted diffeomorphism $\overline{\psi}$ that makes the following diagram commute:
\begin{equation}\label{eq:psi-liftable}
\xymatrix{
	U \times  \S^{ n_S  } \times [0,1)
	\ar[r]^{\overline{\psi}}
	\ar[d]_Q&
	D_S \times[0,1)
	\ar[d]_{\mathscr{L}_S}\\
	U\times  \D^{n_S+1}
	\ar[r]^{\psi}
	&T_S,}
\end{equation}
where $\overline{\psi}(u,\theta,t) = (\psi(u,1/2 \cdot \theta),t).$ The following lemma proves that the same lifting property holds for the charts of $\mathcal{B}$.
\begin{lemma}
\label{lem:appendix:diag}
Let $(U,\phii)$ be a chart of $\mathcal{B}$. Then there exists a foliated embedding 
$\overline{\phii}
\colon ( U \times  \S^{ n_S  }  \times[0,1),\mathcal{K} \times
{\mathcal{G}}_{S} \times \mathcal{I}) \to
(D_S\times[0,1),\mathcal{K} \times \mathcal{I})
$
making the following diagram commute:
$$
\xymatrix{
	U\times  \S^{ n_S  } \times [0,1)
	\ar[r]^{\overline{\phii}}
	\ar[d]_Q&
	D_S \times[0,1)
	\ar[d]_{\mathscr{L}_S}\\
	U\times  \D^{n_S+1}
	\ar[r]^{\phii}
	&T_S.}
$$
\end{lemma}
\begin{proof}
Take a chart $(U,\psi)\in\mathcal{A}$. Notice that the lifted map $\overline{\psi}$ (see \eqref{eq:psi-liftable}) is an embedding and the composition $h=\psi^{-1}\circ\phii$ is a diffeomorphism. 
By the first part of Lemma \ref{lem:appendix:borde}, $h$ can be lifted to the diffeomorphism
$H \colon
U \times \S^{n_S}\times [0,1) \to U \times \S^{n_S}\times [0,1)$
and defining 
$\overline{\phii} =\overline{\psi} \rondp H $ we get the following commutative diagram:
$$
\xymatrix{
	U \times \S^{n_S}\times [0,1)\ar[d]_Q\ar[r]_H\ar@/^1.5pc/[rr]^{\overline{\phii}}  
	& U \times \S^{n_S}\times [0,1)\ar[r]_{\overline{\psi}}\ar[d]_Q 
	& D_S\times[0,1)\ar[d]^{\mathscr{L}_S}\\
	U\times \D^{n_S +1}\ar[r]^h \ar@/_1.5pc/[rr]_{\phii}
	& U\times \D^{n_S +1}\ar[r]^{\psi} & T_S.}
$$
The restriction
$ \overline{\phii} \colon (U\times \S^{ n_S  } \times (0,1) ,
\mathcal{K} \times
{\mathcal{G}}_{S} \times \mathcal{I}) \to
(D_S\times (0,1) , \mathcal{K} \times \mathcal{I})
$
is a foliated embedding since
the restrictions
$
Q \colon (U \times \S^{ n_S  } \times (0,1) , \mathcal{K} \times
{\mathcal{G}}_{S} \times \mathcal{I})
\to
(U \times \R^{n_S +1} \backslash \{ 0 \} , \mathcal{K}\times
{\mathcal{E}}_{S} ) $
and
$
{\mathscr{L}}_{S} \colon (D_S \times (0,1)) \to
(T_S\backslash S,
\mathcal{K})
$
are foliated diffeomorphisms (cf. 6.1).
Now, it suffices to apply the Lemma  \ref{lem:appendix:embed} to
$\overline{\phii}$. The uniqueness of $\overline{\phii}$ follows by density.
\end{proof}

\subsection{Proof of Proposition \ref{lema:struct}}
For each $(U,\phii) \in \mathcal{B}$ we define
$$
\overline{\overline{\phii}} \colon (U \times \S^{ n_S  },
\mathcal{K}
\times {\mathcal{G}}_{S}) \to (D_S, \mathcal{K})
$$
by
$
\overline{\overline{\phii}} (u,\theta) = \pr
\overline{\phii}(u,\theta,0),
$
where $\pr \colon D_{S }\times[0,1) \to D_S$ is the canonical
projection and $\overline{\phii}$ comes from Lemma \ref{lem:appendix:diag}.
Notice that $\overline{\overline{\phii}}$ is a foliated embedding. Since ${\mathscr{L}}_{S} (z,0) =
\tau_S(z)$ for each $z \in D_S$, we have
$$
\tau_S\overline{\overline{\phii}} (u,\theta)
=
\tau_S\pr
\overline{\phii}(u,\theta,0)
=
{\mathscr{L}}_{S} (\pr
\overline{\phii}(u,\theta,0),0)
=
{\mathscr{L}}_{S} \overline{\phii}(u, \theta , 0)
=
\phii Q(u,\theta,0)
=
\phii(u)= u.
$$
We conclude that $\mathcal{C} = \{
(U,\overline{\overline{\phii}}) \ | \ (U,\phii) \in \mathcal{B}) \}$
is an atlas of the sphere bundle ${\tau}_{s} \colon D_S \to
S$.
By construction, the  structure group preserves
${\mathcal{G}}_{S}$.
It remains to prove that it also belongs to the orthogonal group
$
O(n_S+1).
$

Consider $(U_{i},\overline{\overline{\phii_{i}}}),
(U_{j},\overline{\overline{\phii_{j}}}) \in \mathcal{C} $.
In the commutative diagram
$$
\xymatrix@+2pc{
U_{i}\cap U_{j}  \times  \S^{ n_S  }  \times [0,1)
\ar[r]^{\overline{\phii_{i}}^{-1}\rondp \overline{\phii_{j}}}
\ar[d]_Q
& U_{i}\cap U_{j} \times  \S^{ n_S  }  \times[0,1)
\ar[d]_Q\\
U_{i}\cap U_{j}\times  \D^{n_S+1} \ar[r]^{\phii_{i}^{-1} \rondp
\phii_{j}} 
& U_{i}\cap U_{j} \times  \D^{n_S+1} ,\\
}$$
(cf. Lemma \ref{lem:appendix:diag})
the two horizontal rows are foliated diffeomorphisms.
The top map is determined by the bottom map as described in the
proof of Lemma \ref{lem:appendix:borde}. So  for each $(u,\theta,0) \in U_{i}\cap
U_{j}  \times  \S^{ n_S  }   \times
[0,1)$ we get
$$
\overline{\phii_{i}}^{-1} \rondp \overline{\phii_{j}} (u,\theta,0)
=
(u, G_u (\theta)/||G_u (\theta)||,0).
$$
This equality yields
$$
\overline{\overline{\phii_{i}}}^{-1} \rondp
\overline{\overline{\phii_{j}}} (u,\theta)
=
(u, G_u (\theta)/||G_u (\theta)||).
$$
The linear map $G_{u}$ is the differential of the smooth map
$
f_{i,j} (u,-\! \! -) \colon
\D^{n_S+1} \to\D^{n_S+1},
$ where
$\phii_{i}^{-1} \rondp \phii_{j} (u,\theta) =
(u,f_{i,j}(u,\theta)).$
Since $(U_{i},\phii_{i}) \in \mathcal{B}$ then
$|\phii_{i}(u,v)| = |v|$ for each $(u,v) \in U \times
\D^{n_S+1} $. So  we get that $G_{u} \in O(n_S+1)$. We have finished, since
$$
\overline{\overline{\phii_{i}}}^{-1} \rondp
\overline{\overline{\phii_{j}}} (u,\theta)
=
(u, G_u (\theta)).\qed
$$

\bibliography{nirebib}

\providecommand{\bysame}{\leavevmode\hbox to3em{\hrulefill}\thinspace}
\providecommand{\MR}{\relax\ifhmode\unskip\space\fi MR }
\providecommand{\MRhref}[2]{%
  \href{http://www.ams.org/mathscinet-getitem?mr=#1}{#2}
}
\providecommand{\href}[2]{#2}
\begin{thebibliography}{RPSAW09}

\bibitem[AL92]{suso}
J.~A. Alvarez~L{\'o}pez, \emph{The basic component of the mean curvature of
  {R}iemannian foliations}, Ann. Global Anal. Geom. \textbf{10} (1992), no.~2,
  179--194. \MR{1175918 (93h:53027)}

\bibitem[AR17]{alexandrino}
M.~M. Alexandrino and M.~Radeschi, \emph{Closure of singular foliations: the
  proof of {M}olino's conjecture}, Compos. Math. \textbf{153} (2017), no.~12,
  2577--2590. \MR{3705298}

\bibitem[BM93]{boualem-molino}
H.~Boualem and P.~Molino, \emph{Mod\`eles locaux satur\'es de feuilletages
  riemanniens singuliers}, C. R. Acad. Sci. Paris S\'er. I Math. \textbf{316}
  (1993), no.~9, 913--916. \MR{1218287 (94b:53058)}

\bibitem[Car84]{carriere}
Y.~Carri\`ere, \emph{Flots riemanniens}, Ast\'erisque (1984), no.~116, 31--52,
  Transversal structure of foliations (Toulouse, 1982). \MR{755161}

\bibitem[CE97]{cairns-escobales}
G.~Cairns and R.~H. Escobales, Jr., \emph{Further geometry of the mean
  curvature one-form and the normal plane field one-form on a foliated
  {R}iemannian manifold}, J. Austral. Math. Soc. Ser. A \textbf{62} (1997),
  no.~1, 46--63. \MR{1427628 (97m:53050)}

\bibitem[Dom98]{dominguez}
D.~Dom{\'{\i}}nguez, \emph{Finiteness and tenseness theorems for {R}iemannian
  foliations}, Amer. J. Math. \textbf{120} (1998), no.~6, 1237--1276.
  \MR{1657170 (99k:53045)}

\bibitem[Ghy84]{ghys}
{\'E}.~Ghys, \emph{Feuilletages riemanniens sur les vari\'et\'es simplement
  connexes}, Ann. Inst. Fourier (Grenoble) \textbf{34} (1984), no.~4, 203--223.
  \MR{766280 (86c:57025)}

\bibitem[Hae73]{haefliger73}
A.~Haefliger, \emph{Sur les classes caract\'eristiques des feuilletages},
  S\'eminaire {B}ourbaki, 24\`eme ann\'ee (1971/1972), {E}xp. {N}o. 412,
  Springer, Berlin, 1973, pp.~239--260. Lecture Notes in Math., Vol. 317.
  \MR{0420655}

\bibitem[Hae80]{haefliger80}
\bysame, \emph{Some remarks on foliations with minimal leaves}, J. Differential
  Geom. \textbf{15} (1980), no.~2, 269--284 (1981). \MR{614370 (82j:57027)}

\bibitem[HH86]{hectorA}
G.~Hector and U.~Hirsch, \emph{Introduction to the geometry of foliations.
  {P}art {A}}, second ed., Aspects of Mathematics, vol.~1, Friedr. Vieweg \&
  Sohn, Braunschweig, 1986, Foliations on compact surfaces, fundamentals for
  arbitrary codimension, and holonomy. \MR{881799}

\bibitem[KT83]{kamber-tondeur-basic}
F.~W. Kamber and P.~Tondeur, \emph{Foliations and metrics}, Differential
  geometry ({C}ollege {P}ark, {M}d., 1981/1982), Progr. Math., vol.~32,
  Birkh\"auser Boston, Boston, MA, 1983, pp.~103--152. \MR{702530 (85a:57017)}

\bibitem[KT84]{kamber-tondeur-dual}
\bysame, \emph{Duality theorems for foliations}, Ast\'erisque (1984), no.~116,
  108--116, Transversal structure of foliations (Toulouse, 1982). \MR{755165
  (86i:57033)}

\bibitem[Mas92]{masa}
X.~Masa, \emph{Duality and minimality in {R}iemannian foliations}, Comment.
  Math. Helv. \textbf{67} (1992), no.~1, 17--27. \MR{1144611 (93g:53040)}

\bibitem[Mat12]{mather}
J.~Mather, \emph{Notes on topological stability}, Bull. Amer. Math. Soc. (N.S.)
  \textbf{49} (2012), no.~4, 475--506. \MR{2958928}

\bibitem[Mol88]{molino}
P.~Molino, \emph{Riemannian foliations}, Progress in Mathematics, vol.~73,
  Birkh\"auser Boston, Inc., Boston, MA, 1988, Translated from the French by
  Grant Cairns, With appendices by Cairns, Y. Carri{\`e}re, {\'E}. Ghys, E.
  Salem and V. Sergiescu. \MR{932463 (89b:53054)}

\bibitem[MW06]{miquel-wolak}
V.~Miquel and R.~A. Wolak, \emph{Minimal singular {R}iemannian foliations}, C.
  R. Math. Acad. Sci. Paris \textbf{342} (2006), no.~1, 33--36. \MR{2193392
  (2006i:53029)}

\bibitem[Noz10]{nozawa-manuscripta}
H.~Nozawa, \emph{Rigidity of the \'{A}lvarez class}, Manuscripta Math.
  \textbf{132} (2010), no.~1-2, 257--271. \MR{2609297}

\bibitem[Noz12]{nozawa-arxiv}
\bysame, \emph{{Haefliger cohomology of Riemannian foliations}}, ArXiv e-prints
  (2012).

\bibitem[NRP14]{nozawa-royo}
H.~Nozawa and J.~I. Royo~Prieto, \emph{Tenseness of {R}iemannian flows}, Ann.
  Inst. Fourier (Grenoble) \textbf{64} (2014), no.~4, 1419--1439. \MR{3329668}

\bibitem[RP]{royotesis}
J.~I. Royo~Prieto, \emph{Estudio cohomol\'ogico de flujos riemannianos}, Ph.D.
  Thesis, University of the Basque Country UPV/EHU, 2003,
  http://www.ehu.eus/joseroyo/pdf/tesis.pdf.

\bibitem[RPSAW05]{topinthebic}
J.~I. Royo~Prieto, M.~Saralegi-Aranguren, and R.~Wolak, \emph{Top-dimensional
  group of the basic intersection cohomology for singular {R}iemannian
  foliations}, Bull. Pol. Acad. Sci. Math. \textbf{53} (2005), no.~4, 429--440.
  \MR{2214932 (2006m:53042)}

\bibitem[RPSAW08]{tope1}
\bysame, \emph{Tautness for {R}iemannian foliations on non-compact manifolds},
  Manuscripta Math. \textbf{126} (2008), no.~2, 177--200. \MR{2403185
  (2010e:53052)}

\bibitem[RPSAW09]{tope2}
J.~I. Royo~Prieto, M.~Saralegi-Aranguren, and R.~Wolak, \emph{Cohomological
  tautness for {R}iemannian foliations}, Russ. J. Math. Phys. \textbf{16}
  (2009), no.~3, 450--466. \MR{2551892 (2010j:53043)}

\bibitem[Rum79]{rummler}
H.~Rummler, \emph{Quelques notions simples en g\'eom\'etrie riemannienne et
  leurs applications aux feuilletages compacts}, Comment. Math. Helv.
  \textbf{54} (1979), no.~2, 224--239. \MR{535057 (80m:57021)}

\bibitem[Ste74]{stefan}
P.~Stefan, \emph{Accessible sets, orbits, and foliations with singularities},
  Proc. London Math. Soc. (3) \textbf{29} (1974), 699--713. \MR{0362395 (50
  \#14837)}

\bibitem[Sul79]{sullivan}
D.~Sullivan, \emph{A homological characterization of foliations consisting of
  minimal surfaces}, Comment. Math. Helv. \textbf{54} (1979), no.~2, 218--223.
  \MR{535056}

\bibitem[Sus73]{sussmann}
H.~J. Sussmann, \emph{Orbits of families of vector fields and integrability of
  distributions}, Trans. Amer. Math. Soc. \textbf{180} (1973), 171--188.
  \MR{0321133 (47 \#9666)}

\bibitem[Tho69]{thom}
R.~Thom, \emph{Ensembles et morphismes stratifi\'es}, Bull. Amer. Math. Soc.
  \textbf{75} (1969), 240--284. \MR{0239613 (39 \#970)}

\bibitem[Ver71]{verona}
A.~Verona, \emph{Le th\'eor\`eme de de {R}ham pour les pr\'estratifications
  abstraites}, C. R. Acad. Sci. Paris S\'er. A-B \textbf{273} (1971),
  A886--A889. \MR{0290375 (44 \#7559)}

\bibitem[Wol89]{wolak-subalgebras}
R.~A. Wolak, \emph{Maximal subalgebras in the algebra of foliated vector fields
  of a {R}iemannian foliation}, Comment. Math. Helv. \textbf{64} (1989), no.~4,
  536--541. \MR{1022996}

\end{thebibliography}

\bibliographystyle{nireamsalpha}

\end{document}